\documentclass[final]{siamltex}

\usepackage{graphicx}        
\usepackage{amsfonts}
\usepackage{amsmath}
\usepackage{amssymb}

\newcommand{\RR}{\ensuremath{\mathbb R}}
\newcommand{\CC}{\ensuremath{\mathbb C}}

\newcommand{\ds}{\displaystyle}

\newcommand{\uu}{v}

\newcommand{\ff}{p}

\newcommand{\hu}{\hat{u}}
\newcommand{\huu}{\hat{v}}

\newcommand{\hhh}{\hat{q}}
\newcommand{\hg}{\hat{g}}

\newcommand{\cR}{{\cal R}}

\begin{document}

\title{A new Algorithm Based on Factorization for Heterogeneous
  Domain Decomposition}

\author{Martin J. Gander\thanks{Section de math\'ematiques,
    Universit\'e de Gen\`eve, 2-4 rue du Li\`evre, CP 64, CH-1211
    Gen\`eve 4, Switzerland.  \texttt{Martin.Gander@unige.ch}} \and
  Laurence Halpern\thanks{LAGA, Universit\'e Paris XIII, 99 Avenue
    J.-B. Cl\'ement, 93430 Villetaneuse, France.
    \texttt{halpern@math.univ-paris13.fr}} \and V\'eronique
  Martin\thanks{LAMFA UMR-CNRS 7352, Universit\'e de Picardie Jules
    Verne, 33 Rue St. Leu, 80039 Amiens, France.
    \texttt{veronique.martin@u-picardie.fr}}}

\maketitle

\begin{abstract}
Often computational models are too expensive to be solved in the
entire domain of simulation, and a cheaper model would suffice away
from the main zone of interest. We present for the concrete example of
an evolution problem of advection reaction diffusion type a
heterogeneous domain decomposition algorithm which allows us to
recover a solution that is very close to the solution of the fully
viscous problem, but solves only an inviscid problem in parts of the
domain. Our new algorithm is based on the factorization of the
underlying differential operator, and we therefore call it
factorization algorithm.  We give a detailed error analysis, and show
that we can obtain approximations in the viscous region which are much
closer to the viscous solution in the entire domain of simulation than
approximations obtained by other heterogeneous domain decomposition
algorithms from the literature.
\end{abstract}



\section{Introduction}

The coupling of different types of partial differential equations is
an active field of research, since the need for such coupling arises
in various applications. A first main area is the simulation of
complex objects, composed of different materials, which are naturally
modeled by different equations; fluid-structure interaction is a
typical example, and many techniques have been developed for this
type of coupling problems, see for example the book
\cite{Morand:1995:FSI}, or the review on the immersed boundary method
\cite{Mittal:2005:IBM}, and \cite{Deparis:2007:HDD} for domain
decomposition coupling techniques. A very important area of
application is the simulation of the cardiovascular system
\cite{Quarteroni:2009:CMA}. A second main area is when homogeneous
objects are simulated, but the partial differential equation modeling
the object is too expensive to solve over the entire object, and a
simpler, less expensive model would suffice in most of the object to
reach the desired accuracy; air flow around an airplane is a typical
example, where viscous effects are important close to the airplane,
but can be neglected further away, see the early publication
\cite{Glowinski:1988:OCV}, and also \cite{coclici:2001:AOA} and the
references therein. An automatic approach for neglecting the diffusion
in parts of the domain is the $\chi$-formulation, see
\cite{Lai:1998:ADE} \cite{brezzi:1989:ASA}, and there are also
techniques based on virtual control, originating in
\cite{Periaux:1988:OCV}, see \cite{Agoshkov:2006:OCI} for the case
with overlap, and \cite{Gervasio:2001:HCB} for the case without, and
also \cite{Discacciati:2010:OTC} for virtual control with variational
coupling conditions.  A third emerging area is the coupling of
equations across dimensions, for example the blood flow in the artery
can be modeled by a one dimensional model, but in the heart, it needs
to be three dimensional, see for example \cite{Formaggia:2001:OTC}.
All these techniques have become known in the domain decomposition
community under the name heterogeneous domain decomposition methods, a
terminology sparked by the review \cite{Quarteroni:1992:HDD}, and the
literature has become vast in this field.

We are interested in this paper in the second situation, where the
motivation for using different equations comes from the fact that we
would like to use simpler, less expensive equations in areas of the
domain where the full model is not needed, and we use as our guiding
example the advection reaction diffusion equation.  We are in
principle interested in the fully viscous solution, but we would like
to solve only an advection reaction equation for computational savings
in part of the domain. Coupling conditions for this type of problem
have been developed in the seminal paper \cite{Gastaldi:1989:OTC}, but
with the first situation described above in mind, i.e. there is indeed
a viscous and an inviscid physical domain, and the coupling conditions
are obtained by a limiting process as the viscosity goes to zero, see
also \cite{Gastaldi:1990:OCT}, and
\cite{barbu2000class,Coclici:2001:AHD} for an innovative correction
layer, and \cite{MR1746243} for the steady case.

Dubach developed in his PhD thesis \cite{Dubach:1993:CRE} coupling
conditions based on absorbing boundary conditions, and such conditions
have been used in order to define heterogeneous domain decomposition
methods in \cite{Gander:2001:OSC}. A fundamental question however in
the second situation described above is how far the solution obtained
from the coupled problem is from the solution of the original, more
expensive one on the entire domain. A first comparison of different
transmission conditions focusing on this aspect appeared in
\cite{Gander:2006:ADP}.  In \cite{Gander:2009:VPW}, coupling
conditions were developed for stationary advection reaction diffusion
equations in one spatial dimensions, which lead to solutions of the
coupled problem that can be exponentially close to the fully viscous
solution, and rigorous error estimates are provided. The coupling
conditions are based on the factorization of the differential
operator, see also \cite{Loheac:1993:PAO}, and the exact factorization
can be used in this one dimensional steady case.  We study in this
paper time dependent advection reaction diffusion problems, where the
exact factorization of the differential operator can not be used any
more, due to the non-local nature of the factors, and new ideas are
needed in order to obtain better coupling conditions than the
classical ones developed for situation one, where the domains are
really physically different.

We present in Section \ref{Sec2} our new factorization algorithm. In
Section \ref{sec:wpalgo1} and \ref{sec:wpalgo2} we give a detailed
error analysis of our algorithm, and prove asymptotic error estimates
when the viscosity is becoming small. In Section \ref{SecNum} we
present numerical experiments which show that our theoretical error
estimates are sharp, and that the new factorization algorithm gives
approximate solutions which are one order of magnitude more accurate
in the viscous region than the best heterogeneous domain decomposition
methods known from the literature. 

\section{A new coupling algorithm based on factorization}\label{Sec2}

We now explain how the factorization technique that led to coupling
conditions of excellent quality for one dimensional problems in
\cite{Gander:2009:VPW} can be used to obtain a new coupling algorithm
for evolution problems which we will call factorization algorithm.

\subsection{Model problem} 

We consider the time dependent advection reaction diffusion equation 
\begin{equation}\label{eqadvisq}
  \begin{array}{rcll}
  {\cal L}_{ad}u:=\displaystyle\partial_t u
      -\nu \partial_{x}^2u +
      a\partial_x u+cu & = & f  & 
      \mbox{in $(-L_1,L_2)\times(0,T)$},\\
   {\cal B}_1u(-L_1,\cdot) & = & g_1\quad & \mbox{on $(0,T)$}, \\
   {\cal B}_2u(L_2,\cdot) & = & g_2 & \mbox{on $(0,T)$},\\
   u(x,0) & = & h & \mbox{in $(-L_1,L_2)$},
  \end{array}
\end{equation}
where $a$ is the velocity field, $\nu>0$ is the viscosity, and $c>0$
is a reaction term. The ${\cal B}_j$, $j=1,2$ are suitable boundary
operators, representing Dirichlet or absorbing boundary conditions.
We consider the following choice of these operators,
depending on the sign of the advection
term $a$,
\begin{equation}\label{eq:defbo}
\begin{array}{|c||c|c|}
\hline
& {\cal B}_1& {\cal B}_2\\
\hline
a > 0 & Id & \partial_t +a\partial_x +c\\
\hline
a <0 & Id & Id\\
\hline
\end{array}
\end{equation}
In the case $a > 0$, the flow is given at the inflow boundary, and an
absorbing boundary condition is prescribed at the outflow
boundary. This can be compared to the situation of the tail of a wing,
where the flow goes from the complicated model region into the
simplified model region. For negative $a$, the flow is prescribed at
the inflow and outflow boundary, which can be compared to the
situation of the front of the wing, where the flow goes from the
simplified model region into the complicated model region, and a
boundary layer forms.

We assume that the initial data $h$ is compactly supported in $
(-L_1,0)$. The forcing term $f$ is compactly supported in
$(-L_1,L_2)\times(0,T]$, and the boundary values $g_1$ and $g_2$ are
compactly supported in $(0,T]$ .  Regularity and compatibility
conditions on the data need to be enforced to have a sufficiently
regular solution, see Section \ref{sec:wpalgo1}.

\subsection{The new algorithm based on factorization}

Using Nirenberg's factorization, we can factor the advection-diffusion
operator into a product of two evolution operators in opposite $x$
directions,
$$
  -\nu \partial_{x}^2 + a \partial_{x} + c +\partial_{t} =  -\nu(\partial_{x} +{\cal P}_+(\partial_t))\,(\partial_{x} +{\cal P}_-(\partial_t)) \pmod{{\cal C}^\infty}.
$$
This factorization has been used to design absorbing boundary
conditions and paraxial equations for hyperbolic problems, see
\cite{BEHJ}. For parabolic problems, Nataf and coauthors
\cite{Loheac:1993:PAO,MR1060603} computed approximations of $u$ via a
double sweep, and also obtained transmission conditions for Schwarz
domain decomposition methods \cite{MR1314997}, which led to the new
class of optimized Schwarz methods, see \cite{Gander:2006:OSM} for an
overview. The same factorization can also be used to obtain incomplete
LU preconditioners \cite{gander2000ailu,gander2005incomplete}, and is
the underlying mathematical structure of the recently developed
sweeping preconditioner \cite{engquist2011sweeping}. We now use this
factorization to define our new factorization algorithm: we define two
subdomains, $\Omega_1=(-L_1,0)$ and $\Omega_2=(0,L_2)$, and want to
couple the advection-diffusion equation in $\Omega_1$ with an
advection equation in $\Omega_2$, defined by the transport operator
${\cal L}_{a}\equiv \partial_t +a \partial_x +c$. Our goal is to
obtain a coupled solution which is as close as possible to the fully
viscous solution of the original problem.

We start with the case $a > 0$, where according to (\ref{eq:defbo})
the exterior boundary condition is ${\cal L}_{a}u(L_2,\cdot)=g_2$. 
Suppose there exists a decomposition ${\cal
  L}_{ad}= \widetilde{{\cal L}}_{ma}{\cal L}_{a}$ with ${\cal L}_{a}$
a transport operator propagating to the right, and $\widetilde{{\cal
    L}}_{ma}$ a transport operator propagating to the left. The
original problem
$$
\left\{
\begin{aligned}
  &\widetilde{{\cal L}}_{ma}{\cal L}_{a}u=f  &&\mbox{ in } \Omega \times (0,T),\\
  & u(-L_1,\cdot)=g_1 &&\mbox{ on } (0,T), \\
  &{\cal L}_{a}u(L_2,\cdot)=g_2 &&\mbox{ on } (0,T),\\
  &u(\cdot,0)=h &&\mbox{ in } \Omega
\end{aligned}
\right.
$$
can then be solved by introducing ${u}_{{ma}}:={\cal L}_{a}u$, and
solving the two problems
\begin{align}\label{algoparfait}
\left\{
\begin{aligned}
  &\widetilde{{\cal L}}_{ma}u_{ma}=f  \mbox{ in } \Omega_2 \times (0,T),\\
  &u_{ma}(L_2,\cdot)=g_2\mbox{ on } (0,T) ,\\
  &u_{ma}(\cdot,0)={\cal L}_{a}u(\cdot,0)\mbox{ in } \Omega_2,
\end{aligned}
\right.
& &
\left\{
\begin{aligned}
  &\widetilde{{\cal L}}_{ma}{\cal L}_{a}u_{ad}=f  \mbox{ in } \Omega_1 \times (0,T),\\
  & u_{ad}(-L_1,\cdot)=g_1\mbox{ on } (0,T), \\
  &{\cal L}_{a}u_{ad}(0,\cdot)=u_{ma}(0,\cdot)\mbox{ on } (0,T),\\
  & u_{ad}(\cdot,0)=h\mbox{ in } \Omega_1,
\end{aligned}
\right.
\end{align}
which leads to $u_{ad}=u_{|\Omega_1}$. Unfortunately, the exact
factorization ${\cal L}_{ad}= \widetilde{{\cal L}}_{ma}{\cal L}_{a}$
is expensive, but we can use an approximation with a remainder,
\begin{equation}\label{eq:factop}
{\cal L}_{ad}=  \frac{\nu}{a^2}({\cal L}_{ma}{\cal L}_{a} -  \cR\,) \mbox{ with }
 \cR\, =(\partial_{t}+c)^2 \mbox{ and } {\cal L}_{ma}= \partial_{t} -a\partial_{x} +c+\frac{a^2}{\nu}.
\end{equation}
The viscous solution $u$ satisfies ${\cal L}_{ma}{\cal
  L}_{a}u=a^2f/\nu +{\cal R }u$, and the algorithm corresponding to
(\ref{algoparfait}) is
\begin{align*}
\left\{
\begin{aligned}
  &{\cal L}_{ma}u_{ma}=\frac{a^2}{\nu}f+\cR\,u   \mbox{ in } \Omega_2 \times (0,T),\\
  &u_{ma}(L_2,\cdot)=g_2\mbox{ on } (0,T),\\
  &u_{ma}(\cdot,0)=f(\cdot,0)+\nu d_x^2h\mbox{ in } \Omega_2,
\end{aligned}
\right.
& &
\left\{
\begin{aligned}
  &{\cal L}_{ad}u_{ad}=f  \mbox{ in } \Omega_1 \times (0,T),\\
  & u_{ad}(-L_1,\cdot)=g_1\mbox{ on } (0,T), \\
  &{\cal L}_{a}u_{ad}(0,\cdot)=u_{ma}(0,\cdot)\mbox{ on } (0,T),\\
  & u_{ad}(\cdot,0)=h\mbox{ in } \Omega_1.
\end{aligned}
\right.
\end{align*}
Since $u$ is unknown to evaluate the remainder, we approximate it by
solving an advection equation, and our new factorization
algorithm is
\begin{align}\label{algoapos}
&\left\{
\begin{aligned}
  &{\cal L}_{a}u_a^k=f \mbox{ in } \Omega_2 \times (0,T),\\
  &u_a^k(0,\cdot)=u_{ad}^{k-1}(0,\cdot) \mbox{ on $(0,T)$},\\
  &u_a^k(\cdot,0)=h\mbox{ in } \Omega_2,
\end{aligned}
\right.\nonumber
\\&\left\{
\begin{aligned}
  &{\cal L}_{ma}u_{ma}^k=\frac{a^2}{\nu}f+\cR\,u_a^k   \mbox{ in } \Omega_2 \times (0,T),\\
  &u_{ma}^k(L_2,\cdot)=g_2 \mbox{ on $(0,T)$},\\
  &u_{ma}^k(\cdot,0)=f(\cdot,0)+\nu d_x^2h\mbox{ in } \Omega_2,
\end{aligned}
\right.
\\&\left\{
\begin{aligned}
  &{\cal L}_{ad}u_{ad}^k=f  \mbox{ in } \Omega_1 \times (0,T),\\
  & u_{ad}^k(-L_1,\cdot)=g_1 \mbox{ on $(0,T)$},\\
  &{\cal L}_{a}u_{ad}^k(0,\cdot)=u_{ma}^k(0,\cdot)\mbox{ on $(0,T)$},\\
  &u_{ad}^k(\cdot,0)=h\mbox{ in } \Omega_1,
\end{aligned}
\right.\nonumber
\end{align}
where we start with a given initial guess
$u_{ad}^{0}(0,\cdot)=g_{ad}^0$.  We will prove well posedness of this
algorithm in Section \ref{sec:wpalgo1}, and give precise error
estimates when $\nu$ is small, which show that the new factorization
algorithm gives one and a half orders of magnitude better solutions in
the viscous subregion than the best other coupling algorithms from the
literature.

When $a<0$, we have the factorization with remainder in reverse order,
${\cal L}_{ad}= \frac{\nu}{a^2}({\cal L}_{a}{\cal L}_{ma} - \cR\,)$,
and now the operator ${\cal L}_{a}$ propagates to the left, and ${\cal
  L}_{ma}$ to the right.  The viscous solution $u$ satisfies ${\cal
  L}_{a}{\cal L}_{ma}u=a^2f/\nu +{\cal R }u$, and introducing
$u_a:={\cal L}_{ma}u$, the algorithm corresponding to
(\ref{algoparfait}) is
\begin{align*}
\left\{
\begin{aligned}
  &{\cal L}_{a}u_{a}=\frac{a^2}{\nu}f+\cR\,u   \mbox{ in } \Omega_2 \times (0,T),\\
  &u_{a}(L_2,\cdot)={\cal L}_{ma}u(L_2,\cdot)\mbox{ on } (0,T),\\
  &u_{a}(\cdot,0)={\cal L}_{ma}u(\cdot,0)\mbox{ in } \Omega_2,
\end{aligned}
\right.
& &
\left\{
\begin{aligned}
  &{\cal L}_{ad}u_{ad}=f  \mbox{ in } \Omega_1 \times (0,T),\\
  & u_{ad}(-L_1,\cdot)=g_1\mbox{ on } (0,T), \\
  &{\cal L}_{ma}u_{ad}(0,\cdot)=u_{a}(0,\cdot)\mbox{ on } (0,T),\\
  & u_{ad}(\cdot,0)=h\mbox{ in } \Omega_1.
\end{aligned}
\right.
\end{align*}
Since $u$ is unknown to evaluate the remainder and the boundary
conditions, we approximate it again by solving an advection equation,
and our new factorization algorithm becomes
\begin{align}\label{algoaneg}
&\left\{
\begin{aligned}
  &{\cal L}_{a}u_a^1=f \mbox{ in } \Omega_2 \times (0,T),\\
  &u_a^1(L_2,\cdot)=g_2\mbox{ on $(0,T)$},\\
  &u_a^1(\cdot,0)=h\mbox{ in } \Omega_2,
\end{aligned}
\right.\nonumber\\
&\left\{
\begin{aligned}
  &{\cal L}_{a}u_{a}^2=\frac{a^2}{\nu}f+\cR\,u_a^1   \mbox{ in } \Omega_2 \times (0,T),\\
  &u_{a}^2(L_2,\cdot)={\cal L}_{ma}u_a^1(L_2,\cdot)\mbox{ on $(0,T)$},\\
  &u_{a}^2(\cdot,0)={\cal L}_{ma}u_a^1(\cdot,0) \mbox{ in } \Omega_2,
\end{aligned}
\right.\\
&\left\{
\begin{aligned}
  &{\cal L}_{ad}u_{ad}=f  \mbox{ in } \Omega_1 \times (0,T),\\
  & u_{ad}(-L_1,\cdot)=g_1\mbox{ on $(0,T)$},\\
  &{\cal L}_{ma}u_{ad}(0,\cdot)=u_{a}^2(0,\cdot)\mbox{ on $(0,T)$},\\
  &u_{ad}(\cdot,0)=h\mbox{ in } \Omega_1,
\end{aligned}
\right.\nonumber
\end{align}
where one could also directly compute ${\cal
  L}_{ma}u_a^1(L_2,\cdot)=2g_2'+(2c+a^2/\nu)g_2-f(L_2,\cdot)$ and
${\cal L}_{ma}u_a^1(\cdot,0)=f(\cdot,0)-2ad_xh+a^2h/\nu$. There is no
iteration for $a<0$ in the algorithm, because the boundary condition
$g_2$ at $x=L_2$ in the first step can not be updated naturally from
the viscous solution $u_{ad}$ in $\Omega_1$. We will study this
algorithm in detail in Section \ref{sec:wpalgo2}, and show that it
gives an order of magnitude better solutions in the viscous subregion
than the other coupling algorithms from the literature.

\subsection{Well-posedness results for advection reaction 
diffusion problems}

We work in the usual Sobolev spaces in time and space, $H^s(0,T)$ and
$H^s(\Omega)$ for $\Omega \subset \RR$, $H^s(\Omega\times(0,T))$ in
the hyperbolic case, and the anisotropic spaces $H^{r,s}(\Omega \times
(0,T))$ in the parabolic case. For clarity, we will add an index
defining time or space in the Sobolev space, for instance $H^s_t\equiv
H^s(0,T)$.  We introduce for any domain $\Omega\subset\RR$ the
anisotropic Sobolev spaces (see \cite{Lions:1968:PAL})
\begin{equation} \label{eq:Sobanis}
  H^{r,s}(\Omega \times (0,T))= L^2(0,T; H^{r}(\Omega))\cap 
    H^{s}(0,T; L^2(\Omega)).
\end{equation}
If $u$ is in $ H^{r,s}(\Omega \times (0,T))$, then for 
any integer $j$ and $k$, we have
\begin{equation} \label{eq:derSobanis}
   \frac{\partial^j}{\partial x^j} \frac{\partial^k}{\partial t^k}u 
     \in H^{\mu,\nu}(\Omega \times (0,T)),\quad \mbox{where}\quad 
       \frac{\mu}{r}=\frac{\nu}{s}=1-(\frac{j}{r}+ \frac{k}{s}).
\end{equation}
We introduce the space $V^{r,s}$ of traces of functions in
$H^{r,s}(\Omega\times(0,T))$ for the half-space $\Omega=\RR^-$ (and
similarly for $\Omega=\RR^+$). Denoting by $f_k$ the trace of the k-th
derivative in time on the initial line, $x\in \RR^-$, and by $g_j$ the
trace of the j-th derivative in space on the boundary $x=0$, $t\in
(0,T)$, the trace space $V^{r,s}$ is defined by
\begin{equation} \label{eq:traces}
  \begin{array}{c}
    V^{r,s} := \left\{ (f_k,g_j) \in 
      \prod_{k < s-\frac{1}{2}}H^{p_k}(\Omega) \times 
      \prod_{j < r-\frac{1}{2}}H^{\mu_j}
      (0,T),\right.\\[2mm]
    p_k=\frac{r}{s}(s-k-\frac{1}{2}),\quad 
      \mu_j=\frac{s}{r}(r-j-\frac{1}{2}),\\         
    \frac{\partial^kg_j}{\partial t^k}(0)= 
        \frac{\partial^jf_k}{\partial x^j}(0),
        \quad\mbox{ if } \frac{j}{r}+ \frac{k}{s} < 
        1 -\frac{1}{2}(\frac{1}{r}+ \frac{1}{s}),\\[2mm]
    \left.\int_0^{\infty}|
        \frac{\partial^jf_k}{\partial x^j}(\sigma^s)-
        \frac{\partial^kg_j}{\partial t^k}(\sigma^r)
        |^2 \frac{d\sigma}{\sigma} < \infty, 
        \quad\mbox{ if } \frac{j}{r}+ \frac{k}{s} = 
        1 -\frac{1}{2}(\frac{1}{r}+ \frac{1}{s})
    \right\}.
  \end{array}
\end{equation}
\begin{theorem}[\cite{Lions:1968:PAL}]\label{th:trace}
  For positive real numbers $r$ and $s$ such that
  $1-\frac{1}{2}(\frac{1}{r}+ \frac{1}{s}) >0$, the trace map
  \begin{equation} \label{eq:tracemap}
     u \mapsto  \left\{
      \{\frac{\partial^ku}{\partial t^k}(x,0)\}_{k < s-\frac{1}{2}},
      \{\frac{\partial^ju}{\partial x^j}(0,t)\}_{j < r-\frac{1}{2}}
                \right\}
  \end{equation}
  is defined and continuous from $H^{r,s}(\Omega \times (0,T))$ onto
  $V^{r,s}$.
\end{theorem}


We start with well-posedness results for the advection equation, by
stating a general result, applicable to ${\cal L}_{a}$ in $\Omega$ or
$\Omega_2$, and ${\cal L}_{ma}$ in $\Omega_2$. To this end, we
introduce ${\cal O}=(x_1,x_2)$ and consider
\begin{equation}\label{eq:genadv}
  {\cal L}_bv:=\partial_t v +b \,\partial_x v +\eta v=p \mbox{ in
  }{\cal O}\times(0,T).
\end{equation}
Let ${\cal M}_{b}$ be the spatial part of the operator ${\cal L}_{b}$,
\textit{i.e.} ${\cal L}_{b}=\partial_t +{\cal M}_{b}$. We denote the
boundary point where the flux enters the domain by $x^-$, the other
boundary point by $x^+$, and define the characteristic time
$\tau(x):=\inf \{t \ge 0, \mbox{s.t. } x-at \notin \bar{{\cal O}}\}$.
If $b>0$, $x^-= x_1$ and $\tau(x)= \frac{x-x_1}{b}$, and if $b<0$,
$x^-=x_2$ and $\tau(x)= \frac{x-x_2}{b}$. Note that $\tau$ is a continuous
function of $x$.
We therefore equip (\ref{eq:genadv}) with initial and boundary
conditions
\begin{equation}\label{eq:genadvdata}
  v(\cdot,0)=h,\quad v(x^-,\cdot)=g.
\end{equation}
The following
well-posedness result can be found in \cite{Metivier}.
\begin{theorem}\label{th:hyperbolic1}
  If $ \ff\in L^2({\cal O}\times(0,T))$, $g \in L^2(0,T)$ and $v_0 \in
  L^2({\cal O})$, then the transport problem
  (\ref{eq:genadv},\ref{eq:genadvdata}) has a unique weak solution $v
  \in L^2_{x,t}$, given by (the characteristic function of $\omega$ in
  $\RR^2$ is denoted by $\mathbf{1}_\omega$)
\begin{equation}\label{eq:addistr}
\begin{aligned}
v(x,t)&= h(x-bt)e^{-\eta t} \mathbf{1}_{t<\tau(x)}
+g(t-\tau(x))e^{-\eta\tau(x)}\mathbf{1}_{t>\tau(x)}  \\
&\qquad+\ds\int_{(t-\tau(x))^+}^t \ff(x-b(t-s),s) e^{-\eta(t-s)}\,ds.
\end{aligned}
\end{equation}
If for some $\gamma>0$ we have $h \in H^\gamma({\cal O})$, $g \in
H^{\gamma}(0,T)$ and $ \ff\in H^{\gamma}({\cal O}\times(0,T))$, with
the compatibility conditions 
\begin{equation}\label{eq:CCgenadv}
d_t^kg(0)=
  ( \sum_{j=0}^{k-1}(-{\cal M}_{b})^j  \partial_t^{k-1-j}
\ff)(x^-,0)+ (-{\cal M}_{b})^k h(x^-)\quad \mbox{for}\quad 0\le k \le \gamma-1,
\end{equation}
then $v\in H^{\gamma}({\cal O}\times(0,T))$ and
$v(x^+,\cdot)\in H^{\gamma}(0,T)$.  Furthermore, we have for $0\le k
\le \gamma$ the estimates
%
\begin{equation}\label{eq:estgenadv}
\eta  \|\partial_t^k v\|_{L^2_{x,t}}^2
+|b|\|\partial_t^k v(x^+,\cdot)\|_{L^2_{t}}^2
\le
\frac{1}{\eta}\| \partial_t^k \ff \|_{L^2_{x,t}}^2 
+\|  \partial_t^k v(\cdot,0)\|_{L^2_{x}}^2
+|b|\|d_t^k g \|_{L^2_{t}}^2.
\end{equation}
\end{theorem}
Similarly, we also use well-posedness results for the advection
reaction diffusion equation
\begin{equation}\label{eq:genadvdiff}
  \begin{array}{rcll}
  {\cal L}_{ad}u:=\displaystyle\partial_t u
      -\nu \partial_{x}^2u +
      a\partial_x u+cu & = & f  & 
      \mbox{in ${\cal O}\times(0,T)$},\\
   {\cal B}_1u(x_1,\cdot) & = & g_1\quad & \mbox{on $(0,T)$}, \\
   {\cal B}_2u(x_2,\cdot) & = & g_2 & \mbox{on $(0,T)$},\\
   u(x,0) & = & h & \mbox{in ${\cal O}$},
  \end{array}
\end{equation}
with boundary operators according to (\ref{eq:defbo}). We define
${\cal M}_{ad}$ to be the spatial part of the operator ${\cal
  L}_{ad}$, \textit{i.e.}  ${\cal L}_{ad}=\partial_t +{\cal M}_{ad}$.
\begin{theorem}\label{th:parabolic}
For $\gamma> 0$, let $f \in H^{2\gamma,\gamma}({\cal O}\times(0,T))$,
$g_1\in H_t^{\gamma+\frac{3}4}$, $g_2\in H_t^{\gamma+\frac{3}4}$ for
negative $a$, and $g_2\in H_t^{\gamma-\frac{1}4}$ for positive $a$, $h\in
H_x^{2\gamma+1}({\cal O})$, with the compatibility conditions for
$0\le k < \gamma-\frac{1}{2}$ and $0\le k' < \gamma-\frac{3}{2}$ for
negative $a$ given by
  \begin{align}\label{eq:ccadho}
1\le j \le 2, \quad  d_t^kg_j(0)=\ds
 (-{\cal M}_{ad})^k h(x_j)+ ( \sum_{j=0}^{k-1}(-{\cal M}_{ad})^j  \partial_t^{k-1-j} f)(x_j,0) ,
\\
\intertext{and for positive $a$ the second compatibility condition
  is replaced by }
  d_t^{k'} g_2 (0) -\nu   \ds
  \sum_{j=0}^{k'-1}(-{\cal M}_{ad})^j  \partial_t^{k'-1-j} \partial_x^2f(x_2,0)-\nu   d_{x}^2 (-{\cal M}_{ad})^{k'} h(x_2) = \partial_t^{k'} f(x_2,0) ,
  \end{align}
  then problem \eqref{eq:genadvdiff} has a unique solution $u$ in
  $H^{2(\gamma+1),\gamma+1}({\cal O} \times (0,T))$.
\end{theorem}
\begin{proof}
Existence and regularity results are well-known for Dirichlet boundary
conditions on both sides, see \cite{Lions:1968:PAL,Metivier}, so we do
not consider the case of negative advection further. In
\cite{Metivier} more precise results with error bounds in $\nu$ for
the hyperbolic equation (see Theorem \ref{th:parabolic2}) can be
found. In the case where $a>0$, due to the absorbing boundary, we need
to modify the proof on the right boundary, and we use a Fourier transform
in time.  A weak solution is obtained by a variational formulation,
like in \cite{martin:2005:osw,bennequin:2009:hba} for instance. The
regularity is obtained as follows: we first modify the boundary
condition in \eqref{eq:genadvdiff} on the right at $x=x_2$ to Dirichlet
data $\tilde{g}_2\in H_t^{\gamma+\frac{3}4}$. Because of the
compatibility conditions on the left, and imposing symmetric
compatibility conditions on $\tilde{g}_2$ on the right, there is a
unique solution $\tilde{u}\in H^{2(\gamma+1),(\gamma+1)}({\cal O} \times
(0,T))$, see \cite{Lions:1968:PAL}. The difference $v=u-\tilde{u}$ is
solution of the homogeneous case of equation \eqref{eqadvisqapos}, but
the boundary condition on the right becomes $ {\cal L}_a
(u-\tilde{u})=q_2:=g_2- {\cal L}_a \tilde{u}$. By the regularity
results above, $q_2$ is in $ H_t^{\gamma-\frac{1}4}$.  To estimate
$v$, we will make use of the Fourier transform. We extend all
functions by $0$ in $\RR_-$, and smoothly into $(T,+\infty)$, and
define
$$
  \hat{v}(\omega)=\frac1{2\pi}\int_\RR e^{-i\omega\,t} v(t) \,dt.
$$
Since the initial value vanishes, the equation is Fourier transformed
in time to
$$
  \widehat{{\cal L}}_{ad}\hat{v}:=-\nu  \partial_{x}^2\hat{v}
    +a \partial_x\hat{v} +(c+i\omega)\hat{v}=0 \mbox{ on ${\cal O}\times \CC$}.
$$ 
This is for each $\omega$ an ordinary differential equation, with
characteristic roots
\begin{equation}\label{deflambda}
  \lambda_+(\omega)=\frac{1}{2\nu}(a+\sqrt{a^2+4\nu (c+i\omega)}),\quad 
  \lambda_-(\omega)=\frac{1}{2\nu}(a-\sqrt{a^2+4\nu (c+i\omega)}),
\end{equation}
with $Re(\lambda_+)>0$ and $Re(\lambda_-)<0$. The general solution is
$$
\hat{v}(x,\omega) = \ell_+(\omega) e^{\lambda_+x}+ \ell_-(\omega) e^{\lambda_- x}.
$$ 
Using the boundary conditions, we then get the solution
\begin{equation}\label{solv2}
\huu(x,\omega)=
  \hhh_2(\omega)\,
  \frac{e^{\lambda_+ (x-x_1)}-e^{\lambda_-(x-x_1)}}
  {\nu \lambda_+^2e^{\lambda_+ (x_2-x_1)} -\nu \lambda_-^2e^{\lambda_- (x_2-x_1)}  },
\end{equation}
where we have used the relation $c+i \omega
+a\lambda_{\pm}=\nu(\lambda_{\pm})^2$.  The value at $x=x_2$ can be
equivalently written as
\begin{equation}\label{solv2reg}
  \hat{v}(x_2, \omega) = \hhh_2(\omega) 
    \frac{e^{-(\lambda_+-\lambda_-) (x_2-x_1)}-1}
    {\nu \lambda_+^2\left((\frac{\lambda_-}{\lambda_+})^2
    e^{-(\lambda_+-\lambda_-) (x_2-x_1)}-1\right)}.
\end{equation}
In order to estimate the regularity of $v(x_2,\cdot)$, we need to
estimate the multiplicative factor on the right for large $\omega
$. We can see from \eqref{deflambda} that $\lambda_+(\omega)\sim
- \lambda_-(\omega)\sim\sqrt{ \frac{i\omega}{\nu}}$. Therefore $
|\hat{v}(x_2, \omega)| \sim\left|\frac{ \hhh_2(\omega) }{\nu
  \lambda_+^2}\right| \sim\left|\frac{ \hhh_2(\omega)
}{\omega}\right|.  $ Since $q_2 \in H_t^{\gamma-\frac{1}4}$, we
conclude that $v(x_2,\cdot)\in H_t^{\gamma+\frac{3}4}$. Then $v$ is
solution of the advection-diffusion equation with Dirichlet boundary
conditions, and the data has sufficient regularity to conclude.
\end{proof}

\section{Properties of the factorization algorithm for positive advection}
\label{sec:wpalgo1}

We consider the advection-diffusion equation in $\Omega=(-L_1,L_2)$ with
Dirichlet boundary condition on the left, and absorbing boundary
condition given by the transport operator on the right (see
\cite{Halpern:1986:ABC}),
\begin{equation}\label{eqadvisqapos}
  \begin{array}{rcll}
  {\cal L}_{ad}u:=\partial_t u       -\nu \partial_{x} ^2u +
      a\partial_{x} u +cu & = & f  & 
      \mbox{in $\Omega\times(0,T)$},\\
   u(-L_1,\cdot) & = & g_1\quad & \mbox{on $(0,T)$}, \\
  {\cal L}_a u(L_2,\cdot) & = & g_2 & \mbox{on $(0,T)$},\\
   u(\cdot,0) & = & h & \mbox{in $\Omega$}.
  \end{array}
\end{equation}
We suppose in this section that $f$ and $g_1$ are compactly supported
in $(0,T]$, that $h$ is compactly supported in $\Omega_1=(-L_1,0)$,
and that for each $t$ the function $f(\cdot,t)$ is compactly supported
in $\Omega$. We further assume that the boundary condition at $L_2$ is
absorbing, that is $g_2=0$.  Therefore the compatibility conditions
are satisfied to any order on both ends of the interval $\Omega$, and
for $ f\in H^{\frac{9}2,\frac{9}4}(\Omega\times(0,T))$, $h\in
H_x^{\frac{11}2}$, and $g_1\in H_t^{3}$, $u$ is defined in
$H^{\frac{13}2,\frac{13}4} (\Omega\times(0,T))$.

\subsection{Well-posedness}\label{sec:defalgo}

The remainder ${\cal R}$ for computing $u_{ma}^k$ in the new factorization
algorithm (\ref{algoapos}) contains two time derivatives, which lead
to an important loss of regularity at each iteration. We will however
see that the error order in $\nu$ can not be improved further after
two iterations, and hence we only study the first two iterations in
detail. We start with the well-posedness of the algorithm.

Algorithm (\ref{algoapos}) starts with an initial guess $g_{ad}^0$ as
boundary condition for $u_a$. We assume that $g_{ad}^0\in H_t^{4}$
and is compactly supported in $(0,T]$. Using that $ f\in
H^{4}(\Omega_2\times(0,T))$, that $h$ vanishes in $\Omega_2$, and that
the compatibility conditions at $x=0$, $t=0$ are satisfied, the
solution of
$$
{\cal L}_{a}u_a^1 =f \mbox{ in } \Omega_2, \quad 
u_a^1(0,\cdot)=g_{ad}^0 , \quad 
u_a^1(\cdot,0)=0\ $$
satisfies $ u_a^1\in H^{4}(\Omega_2\times(0,T))$.

The right hand side for the modified advection equation in
(\ref{algoapos}) is then $f_{ma}^1=\frac{a^2}{\nu}f+\cR\,u _a^1 \in
H^{2}(\Omega_2\times(0,T))$, and solving 
$$
{\cal L}_{ma}u_{ma}^1 =f_{ma}^1 \mbox{ in } \Omega_2, \quad
u_{ma}^1(L_2,\cdot)=0, \quad
u_{ma} ^1(\cdot,0)= 0,
$$ 
the compatibility conditions at $x=L_2$ are again satisfied to any
order, which implies that $u_{ma}^1 \in H^{2}(\Omega_2\times(0,T))$
and $u_{ma} ^1(0,\cdot) \in H^{2}(0,T)$. The latter then becomes the
right boundary data for the advection diffusion problem in $\Omega_1$, 
$$
  {\cal L}_{ad}u_{ad}^1 =f \mbox{ in } \Omega_1, \quad 
    u_{ad}^1(-L_1,\cdot)=g_1, \quad 
    {\cal L}_{a}u_{ad}^1(0,\cdot)=u_{ma}^1(0,\cdot), \quad 
    u_{ad}^1(\cdot,0)=h.
$$
We have seen already that the compatibility conditions on the left are
satisfied, and on the right, at the corner $(0,0)$, with the
regularity of $u_{ma}^1$, the condition $u_{ma}^1(0,0)-\nu d_x^2h(0) =
f(0,0)$ holds, since both sides of this equality vanish.  Since $
f\in H^{\frac{9}2,\frac{9}4}(\Omega\times(0,T))$, $h\in
H^{\frac{11}2}_x$, $g_1\in H^{3}_t$, and $u_{ma}^1(0,\cdot)\in H^{2}_t$,
we obtain $ u_{ad}^1\in H^{\frac{13}2,\frac{13}4}
(\Omega_1\times(0,T))$ and $ u_{ad}^1(0,\cdot) \in H^{3}_t$, and 
at the corner $(0,0)$, we have for $g_{ad}^1:=u_{ad}^1(0,\cdot)$
\begin{equation}\label{CompProperty}
  \begin{aligned}
    g_{ad}^1 (0) =h(0),\quad
    d_tg_{ad}^1(0)+{\cal M}_{ad}h(0)&= f(0,0), \\
    d_t^2  g_{ad}^1 (0)-{\cal M}_{ad}^2h(0)&= \partial_tf(0,0)-{\cal M}_{ad}f(0,0).
  \end{aligned}
\end{equation}
We now start the second iteration with the computation of $u_a^2$,
using $u_a^1(0,\cdot)=g_{ad}^1=u_{ad}^1(0,\cdot)\in H^3_t$. Since $h$
is compactly supported in $\Omega_1$, ${\cal M}_{ad}^ph(0)={\cal
  M}_{a}^ph(0)=0$, and (\ref{CompProperty}) are
appropriate compatibility conditions to compute $u_a^2\in
H^{3}(\Omega_2\times(0,T))$. We define the new right hand side
$$
 f_{ma}^{2}=\frac{a^2}{\nu}f+\cR\,u _a^2 \in H^{1}(\Omega_2\times(0,T)),\quad
$$
and compute the solution of 
$$
{\cal L}_{ma}u_{ma}^2 =f_{ma}^2 \mbox{ in } \Omega_2, \quad
u_{ma}^2(L_2,\cdot)=0, \quad
u_{ma} ^2(\cdot,0)= 0.
$$
We thus obtain $ u_{ma}^2 \in H^{1}(\Omega_2\times(0,T)) $ and $u_{ma}^2
(0,\cdot) \in H^{1}(0,T)$.  
The last step of the second iteration is to compute $u_{ad}^2$
solution of
$$
  {\cal L}_{ad}u_{ad}^2 =f \mbox{ in } \Omega_1, \quad 
  u_{ad}^1(-L_1,\cdot)=g_1, \quad 
  {\cal L}_{a}u_{ad}^1(0,\cdot)=u_{ma}^2(0,\cdot), \quad 
  u_{ad}^1(\cdot,0)=h, 
$$ 
and we need only to satisfy a compatibility condition on the left,
which implies that $u_{ad}^2\in H^{\frac{9}2,\frac{9}4}
(\Omega\times(0,T))$.

\subsection{Error estimates for the factorization algorithm}
\label{subsec:errestim}

We present now asymptotic error estimates for small viscosity $\nu$
when $u$, the viscous solution of (\ref{eqadvisqapos}), is
approximated by $(u_{ad}^k,u_a^k)$, the solution obtained by our new
factorization algorithm (\ref{algoapos}).
We define the error quantities $e_a^k:=u_a^k-u$, $e_{ad}^k:=u_{ad}^k-u$, $e_{ma}^k:=u_{ma}^k-{\cal L}_a u:=u_{ma}^k-u_{ma}$,
and suppose that all our data is ${\cal
  C}^\infty$ in all variables.
 The error equations are 
\begin{equation}\label{errorequations}
 \left\{\begin{array}{l}
{\cal L}_{a}e_a^k = -\nu\partial_{x}^2 u \mbox{ in } \Omega_2,\\
e_a^k(0,\cdot)=e_{ad}^{k-1}(0,\cdot),\\
e_a^k(\cdot,0)=0,
\end{array}\right.\
\left\{\begin{array}{l}
{\cal L}_{ma}e_{ma} ^k= \cR\,e_a^k \mbox{ in } \Omega_2,\\
e_{ma}^k(L_2,\cdot)=0,\\
e_{ma}^k(\cdot,0)= 0,
\end{array}\right.\
\left\{\begin{array}{l}
{\cal L}_{ad}e_{ad}^k =0\mbox{ in } \Omega_1,\\
e_{ad}^k(-L_1,\cdot)=0,\\
{\cal L}_{a}e_{ad}^k(0,\cdot)=e_{ma}^k(0,\cdot),\\
e_{ad}^k(\cdot,0)=0,
\end{array}\right.
\end{equation}
with $e_{ad}^{0}(0,\cdot):= g_{ad}^{0}-u(0,\cdot)$. We need more precise
estimates than those provided by Theorems \ref{th:hyperbolic1} and
\ref{th:parabolic}. First, we state precisely the initial conditions
for all the equations involved:
the parabolic problems in $\Omega$ and $\Omega_1$ will use
$$
\partial_t^ku(\cdot,0)=(-{\cal M}_{ad})^k h\quad \mbox{and}\quad
\partial_t^ku_{ad}^p(\cdot,0)=(-{\cal M}_{ad})^k h,
$$
the forward hyperbolic problem in $\Omega_2$ uses
$$
\partial_t^ku_{a}^p(\cdot,0)=(-{\cal M}_{a})^k h=0,
$$
and the backward hyperbolic problem in $\Omega_2$ uses
$$
\partial_t^ku_{ma}^p(\cdot,0)=
\sum_{j=0}^{k-1}(-{\cal M}_{ma})^j
\partial_t^{k-1-j}f_{ma}^p(\cdot,0)
+(-{\cal M}_{ma})^k u_{ma,0}= 0.
$$
The solution of the exact backward hyperbolic problem in $\Omega_2$ 
has as initial condition
$$
\partial_t^ku_{ma}(\cdot,0)=
\sum_{j=0}^{k-1}(-{\cal M}_{ma})^j
\partial_t^{k-1-j}f_{ma}(\cdot,0)
+(-{\cal M}_{ma})^k u_{ma,0}= 0,
$$
from which we infer the initial values for the errors such that
\begin{align}\label{eq:errordttout} 
\partial_t^ke_{ad}^p(\cdot,0)=0 \mbox{ in $\Omega_1$},\quad
 \partial_t^ke_{a}^p(\cdot,0)=0\mbox{ and } \partial_t^ke_{ma}^p(\cdot,0)=0
\mbox{ in $\Omega_2$}.
\end{align}

We start with estimates for the solution of the advection-diffusion
equation \eqref{eqadvisqapos} with vanishing initial data and
vanishing boundary data $g_1$. A first lemma gives results for the
problem with $g_2=0$, based on energy estimates, and a second lemma
gives estimates where only the right-hand side $f$ is non-zero.
\begin{lemma}\label{lemma:estimeAD1}
  Suppose that $a>0$, and that $h$ vanishes identically in $\Omega$,
  $g_1$ and $g_2$ vanish on $(0,T)$, and that $f$ is in ${\cal
    C}^\infty_0(\Omega\times(0,T])$. Then there is a positive constant
    $C$ such that for any $\nu>0$, and any $k \le \gamma$, the
    solution $ u_1 $ of \eqref{eqadvisqapos} satisfies the estimates
\begin{align}\label{eq:adest2}
  \|\partial_t^k u_1 \|_{L^2_{x,t}}^2
   + \|\partial_t^k u_1(L_2,\cdot)\|_{L^2_{t}}^2
   +\nu\|\partial_t^k\partial_x u_1\|_{L^2_{x,t}}^2
   &\le C \| \partial_t^k f \|_{L^2_{x,t}}^2,\\
    \label{eq:adest3}
   \nu\|\partial_t^k\partial_x^2  u_1 \|_{L^2_{x,t}}
   &\le \| \partial_t^k f\|_{L^2_{x,t}}.
\end{align}
\end{lemma}
\begin{proof}
  Since the compatibility conditions are satisfied, $u_1$ is in
  $H^{\infty}$, and the initial value of $\partial_t^k u_1$ vanishes
  as well. We start with $k=0$:  multiplying the equation by $u_1$ and
  integrating over $\Omega$, taking into account that $u_1$ vanishes at
  $-L_1$ gives
  \begin{align}\nonumber
    \frac12 \frac{d\,}{dt} \|u_1(\cdot,t)\|^2_{L^2_x}&+
      c\|u_1(\cdot,t)\|^2_{L^2_x}+\nu
      \|\partial_xu_1(\cdot,t)\|^2_{L^2_x}\\\nonumber
      &+\frac{a}{2}u_1^2(L_2,t)-\nu (u_1\partial_xu_1)(L_2,t)
      =\int_\Omega f(x,t)u_1(x,t) dx.
  \end{align}
  Using the boundary condition at $L_2$ yields
  \begin{align}\nonumber
    \frac12 \frac{d\,}{dt}( \|u_1(\cdot,t)\|^2_{L^2_x}&+\frac{\nu}{a} u_1^2(L_2,t))
    +c\|u_1(\cdot,t)\|^2_{L^2_x}+\nu
    \|\partial_xu_1(\cdot,t)\|^2_{L^2_x}\\\nonumber
    &+(\frac{a}{2}+\frac{\nu c}{a})u_1^2(L_2,t) =
    \int_\Omega f(x,t)u_1(x,t) dx,
  \end{align}
  and by Cauchy-Schwarz and Young's inequality we obtain
  $$
    \int_\Omega f(x,t)u_1(x,t) dx\le 
    \|u_1(\cdot,t)\|_{L^2_x}\,\|f(\cdot,t)\|_{L^2_x}
    \le \frac{c}{2} \|u_1(\cdot,t)\|^2_{L^2_x} +\frac{1}{2c} 
    \|f(\cdot,t)\|^2_{L^2_x}.
  $$
  Integrating in time over $(0,T)$, and dropping the first term which
  is positive, we obtain, since the initial data vanishes, the
  inequality
  $$
    c \|u_1\|_{L^2_{x,t}}^2+2\nu\|\partial_xu_1 \|_{L^2_{x,t}}^2
    +a \|u_1(L_2,\cdot) \|_{L^2_{t}}^2\le \frac{1}{c} \|f\|_{L^2_{x,t}}^2,
  $$
  which proves \eqref{eq:adest2} for $k=0$. To prove
  \eqref{eq:adest3}, we multiply the equation by $-\partial_x^2 u_1$
  and integrate in $x$,
  \begin{align}\nonumber  
    \nu\|\partial_x^2 u_1(\cdot,t)\|^2_{L^2_x}
    -(\partial_t u_1(\cdot,t), \partial_x^2 u_1(\cdot,t))
    &-a(\partial_x u_1(\cdot,t), \partial_x^2 u_1(\cdot,t))\\\nonumber
    &-c( u_1(\cdot,t), \partial_x^2 u_1(\cdot,t))
    =-(f(\cdot,t), \partial_x^2 u_1(\cdot,t)).
  \end{align}
  An integration by parts leads to 
  \begin{align}\nonumber  
    \nu\|\partial_x^2 u_1(\cdot,t)\|^2_{L^2_x}
    +\frac12\frac{d\,}{dt}\|\partial_x u_1(\cdot,t)\|^2_{L^2_x}
    &+c\|\partial_x u_1(\cdot,t)\|^2_{L^2_x}
    -[\partial_x u_1(\cdot,t)(\partial_t u_1 + \frac{a}2 \partial_x
      u_1 \\\nonumber 
    &+c u_1)(\cdot,t)]_{-L_1}^{L_2}=
    -(f(\cdot,t), \partial_x^2 u_1(\cdot,t)).
  \end{align}
  By the boundary conditions, the boundary terms become
  $$
    a(\partial_x u_1)^2(-L_1,t)+ \frac1{a}(\partial_t u_1+cu_1)^2(L_2,t) > 0.
  $$
  We can now integrate in time and use Cauchy-Schwarz and Young's
  inequality to obtain
  $$
    \nu\|\partial_x^2 u_1 \|^2_{L^2_{x,t}} \le \frac1\nu\|f\|^2_{L^2_{x,t}}.
  $$
  The estimates with the time derivative are obtained by applying the
  equation to $\partial_t^k u$.
\end{proof}

\begin{lemma}\label{lemma:estimeAD2}
Assume that $a>0$. Then there are constants $\bar{\nu}>0$ and $C >0$ such that for $\nu \le
\bar{\nu}$, and for any $g_2\in {\cal C}^\infty_0((0,T])$, the solution $
  u_2 $ of \eqref{eqadvisqapos} with zero data $h$, $f$ and
  $g_1$ satisfies for all $k \le \gamma $ the inequalities
  \begin{align}\label{resADbord}
    \forall x \in [-L_1,L_2],  \ \|\partial_t^k  u_2(x,\cdot)\|_{L^2_t}\le C\nu 
     \| g_2\|_{H^k _t},\\\nonumber
    \|\partial_t^k  u_2\|_{L^2_{x,t}}\le  C\nu^\frac32
      \| g_2\|_{H^k _t},\quad 
    \|\partial_x\partial_t^k  u_2\|_{L^2_{x,t}}\le
      C\nu^\frac12 \| g_2\|_{H^k _t},\\ \nonumber
    \|\partial_x^2\partial_t^k  u_2\|_{L^2_{x,t}} \le
      C\nu^{-\frac12} \| g_2\|_{H^k _t}.
  \end{align}
\end{lemma}
\begin{proof}
We use a Fourier transform argument as in the proof of Theorem
\ref{th:parabolic}, and rewrite \eqref{solv2} as
\begin{equation}\label{solv2b} 
  \hu_2(x,\omega)= \hg_2(\omega)\,e^{\lambda_+ (x-L_2)} \frac{e^{-(\lambda_+-\lambda_-)
      (x+L_1)}-1}{\nu
    \lambda_+^2\left((\frac{\lambda_-}{\lambda_+})^2e^{-(\lambda_+-\lambda_-)
      (L_2+L_1)}-1\right)}.
\end{equation}
Now we search for estimates in $\nu$ that are uniform in $\omega$. We
have for the roots $\lambda_\pm$ the estimates
$$
  |\lambda_-/\lambda_+|<1,\quad Re(\lambda_+-\lambda_-)\,\ge a/\nu,\quad
  |\lambda_+|\,\ge a/\nu.
$$
The numerator in \eqref{solv2b} is bounded by 2. A lower bound for the
denominator is obtained by writing $| \nu \lambda_+^2| \ge
\frac{a^2}{\nu\;}$ together with
$$
   \left|1-\left(\frac{\lambda_-}{\lambda_+} \right)^2
     e^{-(\lambda_+-\lambda_-) (L_2+L_1)}\right| 
     \ge  1-\left|\frac{\lambda_-}{\lambda_+} \right |^2
     e^{-Re(\lambda_+-\lambda_-) (L_2+L_1)}\ \ge\ 1-e^{-\frac{a}{\nu}(L_2+L_1)}.
$$
Inserting these estimates into \eqref{solv2b} gives
$$
  |\hu_2(x,\omega)|\le 
  \frac{2\nu}{a^2}| \hg_2(\omega)| e^{Re \lambda_+ (x-L_2)}\ 
  \ \frac{1}{1-e^{-\frac{a}{\nu}(L_2+L_1)}}\ .
$$
Since $1/|1-\mu|<2$ for $\mu< 1/2$, for $\nu$ sufficiently small so
that $e^{-\frac{a}{\nu}(L_2+L_1)} < 1/2$, we have for any $\omega$,
\begin{equation}\label{eq:est24}
  |\hu_2(x,\omega)|\le \frac{4\nu}{a^2} \ | \hg_2(\omega)|e^{Re \lambda_+ (x-L_2)}
  \le C \nu \ | \hg_2(\omega)|.
\end{equation}
By Parseval's identity, we obtain  
\begin{equation}\label{eq:est25}
  \| u_2(x,\cdot)\|_{L^2(\RR_+)} \le C \nu \ \| g_2\|_{L^2(\RR_+)}.
\end{equation}
Modifying now $ g_2$ to vanish in $[T+\epsilon, \infty)$, the solution
in $(0,T)$ remains unaffected by causality and for any positive
$\epsilon$,
$$
  \| u_2(x,\cdot)\|_{L^2(0,T)} \le C \nu \ \|g_2\|_{L^2(0,T+\epsilon)}.
$$
Since $\epsilon$ is arbitrary, we conclude that 
$$
  \| u_2(x,\cdot)\|_{L^2(0,T)} \le C \nu \ \|g_2\|_{L^2(0,T)}.
$$
From \eqref{eq:est24} we also obtain for all $\omega$
$$
  \|\hu_2(\cdot,\omega)\|_{L^2_x}^2\le \frac{C\nu^2}{Re \lambda_+} 
  \ | \hg_2(\omega)|^2\le C \nu^3 \ | \hg_2(\omega)|^2.
$$
We thus obtain
$$
  \| u_2\|_{L^2_{x,t}} \le C \nu^\frac32 \ \| g_2\|_{L^2_t}.
$$
For the derivative in space, we compute
$$
 \begin{array}{rcl}
   \partial_x\hu_2(x,\omega)&=&  \hg_2(\omega) \frac{\lambda_+e^{\lambda_+ (x+L_1)}-\lambda_-e^{\lambda_- (x+L_1)}}{\nu \lambda_+^2e^{\lambda_+( L_2+L_1)}-\nu \lambda_-^2e^{\lambda_- ( L_2+L_1)}} \\
 &=&
   \hg_2(\omega)\,e^{\lambda_+ (x-L_2)}
  \frac{\frac{\lambda_-}{\lambda_+}e^{-(\lambda_+-\lambda_-) (x+L_1)}-1}{\nu \lambda_+\left((\frac{\lambda_-}{\lambda_+})^2e^{-(\lambda_+-\lambda_-) (L_2+L_1)}-1\right)}.
\end{array}
$$
For small $\nu$, we therefore get as before
$$
  |\partial_x\hu_2(x,\omega)|\le 
  \frac4{\nu |\lambda_+|} | \hg_2(\omega)|\,e^{Re\lambda_+ (x-L_2)}\le 
  \frac4{a}  | \hg_2(\omega)|\,e^{Re\lambda_+ (x-L_2)},
$$
which gives $\|\partial_x \hu_2\|_{L^2(\Omega\times\RR)}^2\le C \nu \|
\hg_2 \|_{L^2(\RR)}^2$, and using the same arguments as before,
$$
 \|\partial_x u_2\|_{L^2_{x,t}}\le C  \nu^\frac12 \| g_2 \|_{L^2_{t}}.
$$
In the same way we find that
$$
 \|\partial_x^2 u_2\|_{L^2_{x,t}}\le C  \nu^{-\frac12} \| g_2 \|_{L^2_{t}}.
$$
\end{proof}
\begin{theorem}\label{th:parabolic2}
Let $a>0$. Then there exist positive constants $C$ and $\bar{\nu}$
such that for any $\nu\le\bar{\nu}$, and for any set of data $h\in
{\cal C}^\infty_0(\Omega)$, $f\in {\cal
  C}^\infty_0(\Omega\times(0,T])$, $g_1\in {\cal C}^\infty_0((0,T])$
  and $g_2\equiv 0$, if $U$ is the solution of the transport equation
\begin{equation}\label{eq:transportomega}
{\cal L}_a U = f \mbox{ in } \Omega\times(0,T), \quad
U(\cdot,0)=h \quad
U( -L_1 ,\cdot)=g_1,
\end{equation} 
then the solution $u$ of the advection-diffusion equation
(\ref{eqadvisqapos}) satisfies the estimate
\begin{align}\label{eq:estimates_for_uU}
   \|\partial_t^k (u-U) \|_{L^2_{x,t}}^2
   &+  \|\partial_t^k  (u-U)(L_2,\cdot)\|_{L^2_{t}}^2
   + \nu \|\partial_t^k\partial_x  (u-U)\|_{L^2_{x,t}}^2 \\\nonumber
   &+ \nu^2\|\partial_t^k\partial_x^2 \,(u-U)\|_{L^2_{x,t}}^2
  \le C  \nu^2 
          \| \partial_t^k  \partial_{x}^2 U \|_{L^2_{x,t}}^2.
\end{align}
Hence $u$ also satisfies the estimate
\begin{align}\label{eq:estimates_for_u}
  \|\partial_t^k u \|_{L^2_{x,t}}^2&+\|\partial_t^k  u(L_2,\cdot)\|_{L^2_{t}}^2
  +  \|\partial_t^k\partial_x  u\|_{L^2_{x,t}}^2\\\nonumber
  &+ \|\partial_t^k\partial_x^2  u\|_{L^2_{x,t}}^2 
  \le C (\| f\|_{H^{k+2}_{x,t}}^2+ \|h\|_{H^{k+2}_x}^2+ \|g_1\|_{H^{k+2}_ {t} }^2).
\end{align}
 \end{theorem}
\begin{proof}
  Since the data is compactly supported, the compatibility conditions
  are automatically satisfied, so that the solutions of the parabolic
  and hyperbolic equations are in ${\cal
  C}^\infty_0(\Omega\times(0,T])$. The estimates
  \eqref{eq:estimates_for_uU} follow directly from Lemma
  \ref{lemma:estimeAD1}, using that $u-U$ is solution of the
  advection-diffusion equation in $\Omega$ with right hand side $\nu
  \partial_{x}^2 U$, and zero initial and boundary conditions on the
  left.  The boundary condition on the right also vanishes, ${\cal
    L}_a(u-U) = -f(L_2,\cdot)=0$, since $f$ is compactly supported in
  $\Omega$.  We define now $B_k= \| f\|_{H^{k}_{x,t}}^2+
  \|h\|_{H^{k}_{x}}^2+ \|g_1\|_{H^{k}_{t}}^2$, and use for $U$ the
  hyperbolic estimates (\ref{eq:estgenadv}) in ${\cal O}=\Omega$,
  $$
    \|\partial_t^k U \|_{L^2_{x,t}}^2 + \|\partial_t^k  U(L_2,\cdot)\|_{L^2_{t}}^2
      \le CB_k.
  $$
  Next, from the advection equation, we deduce that
  $$
    a\partial_xU=-(\partial_t+c)U+f,\quad \mbox{and}\quad
    a^2\partial_x^2U=(\partial_t+c)^2U+(a\partial_x-(\partial_t+c))f,
  $$
  so that
  $$
    \|\partial_t^k\partial_xU\|^2_{L^2_{x,t}} \le C B_{k+1}, \quad
    \|\partial_t^k\partial_x^2U\|^2_{L^2_{x,t}} \le C B_{k+2}, 
  $$
  and from \eqref{eq:estimates_for_uU} we obtain
  \begin{align}\nonumber
    &\|\partial_t^k (u -U)\|_{L^2_{x,t}}^2+\|\partial_t^k  (u-U)(L_2,\cdot)\|_{L^2_{t}}^2
      \le C\nu^2 B_{k+2},\\\nonumber
     &  \|\partial_t^k\partial_x  (u-U)\|_{L^2_{x,t}}^2
      \le C \nu B_{k+2},\quad \|\partial_t^k\partial_x^2  (u-U)\|_{L^2_{x,t}}^2
      \le C B_{k+2}.
  \end{align}
  Writing $u=u-U+U$ gives
  \begin{align*}
    &\|\partial_t^k u \|_{L^2_{x,t}}^2+\|\partial_t^k  u(L_2,\cdot)\|_{L^2_{t}}^2
      \le C(B_k+\nu^2 B_{k+2}),\\
    &\|\partial_t^k\partial_x  u\|_{L^2_{x,t}}^2
      \le C(B_{k+1}+\nu B_{k+2}),\quad
     \|\partial_t^k\partial_x^2  u\|_{L^2_{x,t}}^2\
     \le C  B_{k+2}.
  \end{align*}
  Therefore there is a new constant $C$ and $\bar{\nu}$ such that for
  $\nu\le \bar{\nu}$, \eqref{eq:estimates_for_u} holds.
\end{proof}
 
We now present an improved estimate for the solution of the modified
advection problem in $\Omega_2$.
\begin{theorem}\label{th:hyperbolic2}
  Let $a>0$. Then there exist positive constants $C$ and $\bar{\nu}$
  such that, for $\nu \le \bar{\nu}$, and for any set $p$ compactly
  supported in $\overline{\Omega}_2\times(0,T]$, the solution $\uu$ of
    the initial boundary value problem with modified advection
  $$
  \left\{\begin{array}{rcll}
    {\cal L}_{ma} \uu& = &
        \ff &
        \mbox{in $(0,L_2)\times (0,T)$},\\
    \uu(L_2,\cdot) & = &
        0
         & \mbox{on $(0,T)$,} \\
    \uu(\cdot,0) & = &
           0&
           \mbox{in $(0,L_2)$}\\
    \end{array}\right.
  $$
  satisfies the estimate
%
%
  \begin{equation}\label{eq:adalld1}
    \|\partial_t^k\uu(0,\cdot)\|_{L^2_t}^2 \le \ds C\nu^2
    \big(
      \|\partial_t^k{\ff}(0,\cdot)\|_{L^2_t}^2
      +e^{-2\frac{ aL_2}{\nu}}\|\partial_t^k{\ff}(L_2,\cdot)\|_{L^2_t}^2
      +\nu \| \partial_t^k{\ff}\|_{H^{1}_{x,t}}^2
    \big).
  \end{equation}
\end{theorem}
\begin{proof}
  We first extend $p$ by 0 on $(T,+\infty)$. As in Theorem
  \ref{th:hyperbolic1}, $\uu$ can be obtained using the method of
  characteristics, 
  $$
    \uu(0,t)=\int_{\max(t-\frac{L_2}{a},0)}^t
      {\ff}(a(t-\sigma),\sigma)e^{-\tilde{c}(t- \sigma) }d\sigma,
  $$
  where $\tilde{c}=c+a^2/\nu$.  Integrating by parts, and denoting by
  $d_t$ the characteristic derivative, \textit{i.e.} $d_t
  \phi=\partial_t \phi -a \partial_x \phi$, we obtain
  \begin{align}\nonumber
    \uu(0,t)= \dfrac1{\tilde{c}}\Big(\underbrace{p(0,t)}_I&-
    \underbrace{\begin{cases}
    0 & \text{ for }t<L_2/a,\\
    p(L_2, t- \frac{L_2}{a})e^{-\tilde{c} \frac{L_2}{a}}& \text{ for }t>L_2/a
    \end{cases}}_{II}\\\nonumber
    &-
    \underbrace{\int_{\max(t-\frac{L_2}{a},0)}^t
     d_t{\ff}(a(t-\sigma),\sigma)e^{-\tilde{c}(t- \sigma) }d\sigma
    }_{III}\Big).
  \end{align}
  The norm of the first term is $\|I\|_{L^2_t}=\|p(0,\cdot)\|_{L^2_t}$,
  and the norm of the second term can be estimated as
  $$
    \|II\|_{L^2_t}^2= \ds \int_\frac{L_2}{a}^{ +\infty}
    {\ff}^2(L_2, t- \frac{L_2}{a})e^{-2\tilde{c} \frac{L_2}{a}}\,dt
    \le e^{-2\tilde{c} L_2/a}\|p(L_2,\cdot)\|_{L^2_t}^2.
  $$
  For the norm of the third term, we get
  \begin{align}\nonumber
    \|III\|_{L^2_t}^2&=\int _0^\frac{L_2}{a} \Big( \int _0^{t}
      d_t{\ff}(a(t-\sigma),\sigma)e^{-\tilde{c}(t- \sigma) }\,d\sigma
      \Big)^2\,dt \\\nonumber
     &+  \int _\frac{L_2}{a}^{+\infty} \Big(\int_{t-\frac{L_2}{a}}^t
      d_t{\ff}(a(t-\sigma),\sigma)e^{-\tilde{c}(t- \sigma) }\,d\sigma\Big)^2\,dt,
  \end{align}
  and using the Cauchy-Schwarz inequality, we obtain
  \begin{align}\nonumber
      \|III\|_{L^2_t}^2&\le \ds \dfrac{1}{2\tilde{c} }
      \Big( \int _0^\frac{L_2}{a} \int _0^{t}
      (d_t{\ff})^2(a(t-\sigma),\sigma)\,d\sigma\,dt +
      \int _\frac{L_2}{a}^{+\infty} \int_{t-\frac{L_2}{a}}^t
     (d_t{\ff})^2(a(t-\sigma),\sigma)\,d\sigma\,dt \Big)\\
    & \le \ds \dfrac{1}{2\tilde{c} }  \int _0^{+\infty}\int _0^{L_2}
      (d_t{\ff})^2(x,t )dx \,dt
      \le \dfrac{1}{2\tilde{c} }\|{\ff}\|_{H^{1}_{x,t}}^2,
  \end{align}
  which finally leads to the estimate
  \begin{equation*} 
    \|\uu(0,\cdot)\|_{L^2_t}^2 \le \ds\dfrac{C}{\tilde{c}^2}
    \Big(\|{\ff}(0,\cdot)\|_{L^2_t}^2
    +e^{-2\tilde{c}\frac{ L_2}{a}}\|{\ff}(L_2,\cdot)\|_{L^2_t}^2+
    \dfrac{1}{2\tilde{c} }\| {\ff}\|_{H^{1}_{x,t}}^2 \big).
  \end{equation*}
  Since $\tilde{c} \ge \dfrac{a^2}{\nu}$, we get
  \begin{equation*}  
    \|\uu(0,\cdot)\|_{L^2_t}^2 \le \ds C\nu^2\big(\|{\ff}(0,\cdot)\|_{L^2_t}^2
    +e^{-2\frac{ aL_2}{\nu}}\|{\ff}(L_2,\cdot)\|_{L^2_t}^2+
    \nu \| {\ff}\|_{H^{1}_{x,t}}^2 \big)
  \end{equation*}
  on the enlarged time interval $(0,+\infty)$. Using that the
  extension is vanishing for $t \ge T+\epsilon$ gives the estimate
  \eqref{eq:adalld1} for $k=0$ for any $\epsilon$, and thus on
  $(0,T)$.  Applying \eqref{eq:adalld1} to $\partial_t^k\uu$ gives then
  the general result.
\end{proof}

\begin{theorem}\label{th:theresult}
  Assume that $a>0$, and let $B_k:=\| f\|_{H^{k}_{x,t}}^2+
  \|h\|_{H^{k}_x}^2+\|g_1\|_{H^{k}_t}^2$.  Then there exist positive
  constants $C$ and $\bar{\nu}$ such that, for any data $h\in {\cal
    C}^\infty_0(\Omega_1)$, $f\in {\cal
    C}^\infty_0(\Omega\times(0,T])$, $g_1\in {\cal C}^\infty_0((0,T])$
      and $g_2\equiv 0$, and for any initial guess $g_{ad}^0\in {\cal
        C}^\infty_0((0,T])$ and any $\nu \le \bar{\nu}$, the
        approximation from the new algorithm (\ref{algoapos})
        satisfies the error bounds
  \begin{align}
    \|u-u_a^1\|_{L^2_{x,t}}^2 &\le C(B_2+\|g_{ad}^0 \|_{L^{2}_{t}}^2),&
    \|u-u_{ad}^1\|_{L^2_{x,t}}^2&\le C \nu^{5}(B_5+ \|g_{ad}^0 \|_{H^{3}_{t}} ^2)
    \label{eq:estimiteration1}\\
    \|u-u_a^2\|_{L^2_{x,t}}^2 &\le C\nu^2( B_5+ \nu^2\|g_{ad}^0 \|_{H^{3}_{t}} ^2),&
    \|u-u_{ad}^2\|_{L^2_{x,t}} ^2 &\le C \nu^{8}( B_8+ \nu\|g_{ad}^0 \|_{H^{6}_{t}} ^2).
  \label{eq:estimiteration2}
  \end{align}
\end{theorem}

\begin{proof}
We start with the proof for the first iteration of the new algorithm
(\ref{algoapos}), with the initial guess $g_{ad}^0$, which gives
$e_{ad}^0=g_{ad}^0-u(0,\cdot)$ in the algorithm (\ref{errorequations})
satisfied by the errors:

{\bf Advection:} The error $e_a^1$ is solution of an
advection equation in $\Omega_2$ with right hand side $-\nu
\partial_x^2 u$. Since the initial conditions vanish as described in
\eqref{eq:errordttout}, the hyperbolic estimate in Theorem
\ref{th:hyperbolic1} gives
\begin{equation*}
  \|\partial_t^k e_a^1\|_{L^2_{x,t}}^2 +\|\partial_t^k e_a^1(L_2,\cdot)\|_{L^2_{t}}^2 
  \le C( \nu^2\| \partial_t^k \partial_{x}^2u \|_{L^2_{x,t}}^2
    +\|e_{ad}^0 \|_{H^k_{t}}^2).
\end{equation*}
We bound $\|e_{ad}^0\|_{H^k_{t}}$ by
$\|g_{ad}^0\|_{H^k_{t}}+\|u(0,\cdot)\|_{H^k_{t}}$.  From
\eqref{eq:estimates_for_u}, for small $\nu$, $\| \partial_t^k
u(0,\cdot) \|_{L^2_{x,t}}^2 $ (by the trace theorem) and $\|
\partial_t^k \partial_{x}^2u \|_{L^2_{x,t}}^2 $ are bounded by
$CB_{k+2}$, which gives
\begin{equation}\label{eq:estea}
  \|\partial_t^k e_a^1\|_{L^2_{x,t}}^2 +\|\partial_t^k e_a^1(L_2,\cdot)\|_{L^2_{t}}^2
  \le C(\nu^2B_{k+2} +\|e_{ad}^0\|^2_{H^k_{t}} ) \le
  C(\nu^2B_{k+2} +\|g_{ad}^0 \|_{H^k_{t}}^2+ B_{k+2}).
\end{equation}
This equation gives for $k=0$ and small $\nu$ the first estimate in
\eqref{eq:estimiteration1}.

{\bf Modified advection:} Let
$p_a^1(0,\cdot):=(\partial_t+c)^2e_a^1(0,\cdot)=(\partial_t+c)^2e_{ad}^0$.
We estimate $\partial_t^k e_{ma}^1$ at $x=0$ using Theorem
\ref{th:hyperbolic2},
\begin{align}
  \| \partial_t^k e_{ma}^1(0,\cdot) \|_{L^2_t}^2 & \le \ds C \nu^2 (
   \|\partial_t^k {\ff}_a^1(0,\cdot)\|_{L^2_{t}}^2 
   +e^{-2\frac{ aL_2}{\nu}}\|\partial_t^k{\ff}_a^1(L_2,\cdot)\|_{L^2_t}^2
   +\nu\|\partial_t^k{\ff}_a^1\|_{H^1_{x,t}}^2 )\nonumber\\
    & \le \ds C \nu^2 ( \|e_{ad}^0\|_{H^{k+2}_t}^2 
    +e^{-2\frac{ aL_2}{\nu}}\|\partial_t^k{\ff}_a^1(L_2,\cdot)\|_{L^2_t}^2
    +\nu\|\partial_t^k{\ff}_a^1\|_{H^1_{x,t}}^2 ).
 \label{eq:ema1}  
\end{align}
For the last term on the right, we obtain
\begin{align*}
  \|\partial_t^{k}p_a^1\|_{H^1_{x,t}}^2&=\|\partial_t^{k}p_a^1\|_{L^2_{x,t}}^2+\|\partial_x\partial_t^{k}p_a^1\|_{L^2_{x,t}}^2+
  \|\partial_t^{k+1}p_a^1\|_{L^2_{x,t}}^2\\
  &\le C( \| p_a^1\|_{H^{k+1}(0,T;L^2_{x})}^2+\|\partial_x p_a^1\|_{H^{k}(0,T;L^2_{x})}^2)\\
  &\le C( \| e_a^1\|_{H^{k+3}(0,T;L^2_{x})}^2+ \|\partial_x e_a^1\|_{H^{k+2}(0,T;L^2_{x})}^2).
\end{align*}
Since ${\cal L}_a e_a^1=-\nu \partial_x^2 u$, we get $ \partial_x
e_a^1=-\frac1{a}(\nu \partial_x^2 u +(\partial_t+c)e_a^1)$, and hence
$\| \partial_x e_a^1\|_{H^{k+2}(0,T;L^2_x)}\le C(\nu\| \partial_x^2
u\|_{H^{k+2}(0,T;L^2_x)}+\|e_a^1\|_{H^{k+3}(0,T;L^2_{x})})$. Using
(\ref{eq:estea}), we finally obtain
\begin{align*}
  \|\partial_t^{k}p_a^1\|_{H^1_{x,t}}^2 \le C(\nu^2 B_{k+5}+\|e_{ad}^0\|_{H^{k+3}_t}^2
    +\nu^2\|  \partial_x^2 u\|_{H^{k+2}(0,T;L^2_x)}^2).
\end{align*}
We use now (\ref{eq:estimates_for_u}) which gives $\| \partial_x^2
u\|_{H^{k+2}(0,T;L^2_x)}^2\le CB_{k+4}$ and return to (\ref{eq:ema1}).
From the hyperbolic estimate \eqref{eq:estea}, we see that
$\|\partial_t^k{e}_a^1(L_2,\cdot)\|_{L^2_t}^2$ and
$\|\partial_t^k{e}_a^1\|_{L^2_{x,t}}^2$ are bounded by the same
quantity, and hence we can also use the same bound for $e^{-2\frac{
    aL_2}{\nu}}\|\partial_t^k{\ff}_a^1(L_2,\cdot)\|_{L^2_t}^2$ and
${\nu}\|\partial_t^k{\ff}_a^1\|_{H^1_{x,t}}^2$,
\begin{align}
  \|\partial_t^k e_{ma}^1(0,\cdot) \|_{L^2_t}^2
  \le  \ds &C  (  \nu^2\|e_{ad}^0\|_{H^{k+3}_{t}}^2+\nu^5B_{{k+5}} ). 
  \label{eq:ema1final}  
\end{align}

{\bf Advection-diffusion:} $e_{ad}^1$ is solution of
the advection-diffusion equation with non-zero data only on the
right. Therefore, applying Lemma \ref{lemma:estimeAD2} in $\Omega_1$
with $g_2= e_{ma}^1(0,\cdot)$, we obtain
$$
  \|e_{ad}^{1}\|_{L^2_{x,t}}^2 \le C \nu^{3}  \| e_{ma}^1(0,\cdot) \|_{L^2_t}^2.
$$
Using \eqref{eq:ema1final} with $k=0$ for the last term, we obtain
\begin{equation}\label{eq:estimad1}
  \|e_{ad}^{1}\|_{L^2_{x,t}}^2 \le  C\nu^{5} (\|e_{ad}^0 \|_{H^{3}_{t}}^2+\nu^2B_{{5}})
    \le  C\nu^{5} (B_{{5}}+\|g_{ad}^0 \|_{H^{3}_{t}}^2).
\end{equation}
This equation gives the second estimate in \eqref{eq:estimiteration1},
but we will also need to estimate the value of $e_{ad}^1$ at
$x=0$. Using (\ref{resADbord}) we get
$$
  \| \partial_t^ke_{ad}^1(0,\cdot)\|_{L^2_{t}}^2 \le 
    C\nu^2\| \partial_t^ke_{ma}^1(0,\cdot)\|_{L^2_{t}}^2,
$$
which gives by \eqref{eq:ema1final} again
\begin{align}
  \| \partial_t^ke_{ad}^1(0,\cdot)\|_{L^2_{t}}^2
    \le &C\nu^4 ( \nu^3B_{{k+5}}+\|e_{ad}^0 \|_{H^{k+3}_{t}}^2 ).
  \label{eq:estimad1bord}
\end{align}
In particular, we have
\begin{equation}\label{eq:estimrad1bord}
  \| \cR\,e_{ad}^1(0,\cdot)\|_{L^2_{t}}^2 \le 
    C\nu^4(\nu^3B_{{7}}+\|e_{ad}^0 \|_{H^{5}_{t}}^2).
\end{equation}
We now prove the error estimates for the second iteration:

{\bf Advection:} We again use the hyperbolic estimates for $e_a^2$.
Since the initial values are also vanishing, we obtain as in
\eqref{eq:estea} the estimate
\begin{align}
    \|\partial_t^k e_a^2 \|_{L^2_{x,t}}^2
    +\|\partial_t^k e_a^2(L_2,\cdot) \|_{L^2_{t}}^2
    &\le C(\nu^2B_{k+2}+\|e_{ad}^1(0,\cdot)\|_{H^k_{t}}^2 ).
\end{align}
Inserting \eqref{eq:estimad1bord} we get
\begin{align}
    \|\partial_t^k e_a^2(L_2,\cdot) \|_{L^2_{t}}^2 +\|\partial_t^k e_a^2 \|_{L^2_{x,t}}^2
    &\le C(\nu^2B_{k+2} + \nu^4 ( \nu^3B_{{k+5}} +\|e_{ad}^0 \|_{H^{k+3}_{t}}^2 ))
    \nonumber\\
    &\le C\nu^2( B_{{k+5}}+ \nu^2 \|e_{ad}^0 \|_{H^{k+3}_{t}}^2)
    \label{eq:estea2a}\\
    &\le C\nu^2(B_{{k+5}}+ \nu^2 \|g_{ad}^0\|_{H^{k+3}_{t}}^2+\nu^2 B_{k+5})\nonumber.
\end{align}
The last estimate with $k=0$ gives the first result in
(\ref{eq:estimiteration2}).  

{\bf Modified advection:} Defining
$p_a^2:=\cR\,e_a^2=(\partial_t+c)^2e_a^2$, we obtain using
(\ref{eq:estea2a})
\begin{equation}
  \|\partial_t^k p_a^2(L_2,\cdot) \|_{L^2_{t}}^2+ \| \partial_t^k p_a^2 \|_{L^2_{x,t}}^2
    \le C \nu^2 ( B_{{k+7}}+\|e_{ad}^0 \|_{H^{k+5}_{t}}^2).
     \label{eq:pa2final}
\end{equation}
We estimate $e_{ma}^2$ at $x=0$ by Theorem \ref{th:hyperbolic2},
\begin{equation}
  \| e_{ma}^2(0,\cdot) \|_{L^2_t}^2 \le \ds C\nu^2(
    \|{\ff}_a^2(0,\cdot)\|_{L^2_{t}}^2
    +e^{-2\frac{ aL_2}{\nu}}\|{\ff}_a^2(L_2,\cdot)\|_{L^2_t}^2
    +\nu\|{\ff}_a^2\|_{H^1_{x,t}}^2 ).
    \label{eq:ema2}
\end{equation}
As in the first step, the term at the boundary $x=L_2$ is absorbed in the
volume term, and ${\ff}_a^2(0,\cdot)= \cR\,e^1_{ad}(0,\cdot)$, which
can be estimated by \eqref{eq:estimrad1bord}.  To estimate the term
$\|p_a^2\|_{H^1_{x,t}}$, we proceed as in the first iteration, to obtain
\begin{align}
  \|p_a^2\|_{H^1_{x,t}}^2& \le C( \|e_a^2\|_{H^3(0,T;L^2_{x})}^2
    + \nu^2\| \partial_x^2 u\|^2_{H^2(0,T;L^2_{x})})\nonumber\\
  &\le C (\|e_a^2\|_{H^3(0,T;L^2_{x})}^2+\nu^2 B_4) \nonumber\\  
  &\le C \nu^2( B_{{8}}+ \nu^2 \|e_{ad}^0 \|_{H^{6}_{t}}^2).
  \label{eq:em2ed}
\end{align}
Inserting (\ref{eq:estimrad1bord}) and (\ref{eq:em2ed}) into
\eqref{eq:ema2} we get
\begin{align*}
  \| e_{ma}^2(0,\cdot) \|_{L^2_t}^2
    &\le \ds C\nu^2 ( \nu^4 ( \nu^3B_{{7}}+\|e_{ad}^0 \|_{H^{5}_{t}}^2)
      + \nu^3 ( B_{{8}}+ \nu^2\|e_{ad}^0 \|_{H^{6}_{t}}^2 ))\\ 
  &\le \ds C\nu^5 ( B_{{8}}+\nu\|e_{ad}^0 \|_{H^{6}_{t}}^2).
\end{align*}

{\bf Advection-diffusion:} $e_{ad}^2$ is solution of the
advection-diffusion equation with data only on the right. Therefore,
applying Lemma \ref{lemma:estimeAD2} in $\Omega_1$ with $g_2=
e_{ma}^2(0,\cdot)$, we obtain
\begin{align*}\label{eq:ead2}
 \|e_{ad}^{2}\|_{L^2_{x,t}}^2
    \le C \nu^{3}  \| e_{ma}^2(0,\cdot) \|_{L^2_t}^2  \le \ds
    C\nu^8 ( B_{{8}}+\nu\|e_{ad}^0 \|_{H^{6}_{t}}^2).
\end{align*}
\end{proof}

\section{Properties of the factorization algorithm for negative advection}
\label{sec:wpalgo2}

We consider now the advection-diffusion equation for $a<0$ in
$\Omega=(-L_1,L_2)$ with Dirichlet boundary conditions on both
sides,
\begin{equation}\label{eq:advisqanegdir}
  \begin{array}{rcll}
  {\cal L}_{ad}u:=\displaystyle\partial _t u
 -\nu \partial _{xx} u + a \partial _x u
    +cu & = & f  &
      \mbox{in $\Omega\times(0,T)$},\\
   u(-L_1,\cdot) & = & g_1\quad & \mbox{in $(0,T)$}, \\
   u(L_2,\cdot) & = &g_2& \mbox{in $(0,T)$},\\
   u(\cdot,0) & = & h & \mbox{in $\Omega$}.
  \end{array}
\end{equation}
We suppose again that $f$ and $(g_1,g_2)$ are compactly supported in
$(0,T]$, and that $h$ is compactly supported in $\Omega_1=(-L_1,0)$,
and that for each $t$ the function $f(\cdot,t)$ is compactly supported in
$\Omega$.

\subsection{Well-posedness of the factorization algorithm}

For $u_a^1$, suppose that $f\in H^{3+\frac34}(\Omega\times (0,T))$,
$g_2\in H^{3+\frac34}(0,T)$, $h\in H^{3+\frac34}(\Omega)$, and that
the compatibility conditions \eqref{eq:CCgenadv} are satisfied. Then
we have a unique solution $u_a^1$ in
$W^{3+\frac34}(\Omega_2\times(0,T))$.  For $u_a^2$, using the previous
result, we have $f -\cR\, u_a^1\in
H^{1+\frac34}(\Omega_2\times(0,T))$, and ${\cal L}_{ma} u_a^1\in
H^{2+\frac34}(\Omega_2\times(0,T))$. Therefore the traces at $x=L_2$
and $t=0$ are in $H^{1+\frac34}$ and compatible. Thus \eqref{algoaneg}
defines a unique $u_a^2$ in
$H^{1+\frac34}(\Omega_2\times(0,T))$. Furthermore $u_a^2(0,\cdot)\in
H^{\frac54}(0,T)$. For $u_{ad}$, Theorem \ref{th:parabolic} applies
with $\gamma=\frac54$, and \eqref{algoaneg} defines a unique $u_{ad}$
in $H^{\frac92,\frac94}(\Omega_1\times(0,T))$.  Finally for $u$, using
the regularity assumptions above, $u \in
H^{2(\gamma+1),\gamma+1}(\Omega \times(0,T))$ with
$\gamma=\frac{11}8$.

\subsection{Error estimates for the factorization algorithm}
\label{subsec:errestimaneg}

We need a further lemma in order to obtain our asymptotic estimates.
\begin{lemma}\label{lemma:advdifferreuraneg}
  Suppose $a<0$, and let 
  $g\in L^2(0,T)$. Then there exists a constant $C>0$, such that for
  all $\nu>0$ the solution $v$ of
  \begin{equation}\label{eq:error:ad:aneg}
    \begin{array}{rcll}
    {\cal L}_{ad}{\uu}
       & = & 0 & \mbox{in $\Omega_1\times (0,T)$},\\
    {\uu}(-L_1,\cdot) & = & 0 & \mbox{on $(0,T)$},\\
   (\partial_t-a \partial_x+(\frac{a^2}{\nu}+c)){\uu}(0,\cdot)
       & =&g\hspace{1.95em} & \mbox{on $(0,T)$},\\
   {\uu}(\cdot,0) &=& 0 &\mbox{in $\Omega_1$},
  \end{array}
\end{equation}
satisfies the a priori estimate
$$
  \|\uu\|_{L^2_{x,t}}^2 \le C \nu^2\|g\|_{L^2_t}^2.
$$
\end{lemma}
\begin{proof}
Multiplying the equation by $\uu$, integrating on $(-L_1,0)$ and using
the boundary condition at $x=-L_1$ yields
\begin{equation*}
  \frac{1}{2}\frac{d}{dt}\|\uu(\cdot,t)\|^2_{L^2_x}
  -\frac{|a|}{2}\uu^2(0,t)+\nu\|\partial_x\uu(\cdot,t)\|^2_{L^2_x}
  -\nu \partial_x\uu(0,t)\uu(0,t)+c\|\uu(\cdot,t)\|^2_{L^2_x}=0.
\end{equation*}
Inserting the boundary condition at $x=0$ we obtain
\begin{gather*}
  \frac{1}{2}\frac{d}{dt}(\|\uu(\cdot,t)\|^2_{L^2_x}+\frac{\nu}{|a|}\uu^2(0,t))
  +(\frac{|a|}{2}+\frac{\nu c}{|a|})\uu^2(0,t)\\\hspace{3cm}
  +\nu\|\partial_x\uu(\cdot,t)\|^2_{L^2_x}
  +c\|\uu(\cdot,t)\|^2_{L^2_x}=\frac{\nu}{|a|}g(t)\uu(0,t).
\end{gather*}
Using the inequality
$\nu|g(t)\uu(0,t)|/|a|\le \frac{\nu^2}{|a|^3}g^2(t)+\frac{|a|}{4}\uu^2(0,t)$
and integrating on the time interval $(0,T)$ gives for all $t\in (0,T)$
\begin{equation*}
  c\|\uu\|^2_{L^2_{x,t}} \le \frac{\nu^2}{|a|^3}\int_0^tg^2(\tau)\,d\tau.
\end{equation*}
\end{proof}

We can now prove our main theorem for negative advection.
\begin{theorem}\label{theo_conv_ad:aneg}
  Suppose $a<0$. Then there are positive constants $C $ and
  $\bar{\nu}$ such that for any $h\in {\cal C}^\infty_0(\Omega_1)$,
  $f\in {\cal C}^\infty_0(\Omega\times(0,T])$, $g_1, g_2\in {\cal
  C}^\infty_0((0,T])$, and for any $\nu \le \bar{\nu}$, the
  solution obtained by the new factorization algorithm \eqref{algoaneg}
  satisfies the estimates
  \begin{align}\label{eq:errestglobaneg}
     \|u-u_{a}^1\|_{L^2_{x,t}}&\le C \nu  \|\partial_{x}^2u\|_{L^2_{x,t}},\\
      \quad \|u-u_{ad}\|_{L^2_{x,t}}&\le  C \nu^{2} (\|u\|_{H^{2,2}_{x,t}}
    + \|f(\cdot,0)\|_{H^{2}_{x}}+\|h\|_{H^4_x}
    + \|\partial_x^2u(L_2,\cdot)\|_{L^2_t}),
  \end{align}
  which implies that 
  $$
    \|u-u_{a}^1\|_{L^2_{x,t}} \lesssim \nu, \quad
    \|u-u_{ad}\|_{L^2_{x,t}} \lesssim\nu^{2} .
  $$
\end{theorem}
\begin{proof}
We define the errors $e_a^1:=u_a^1-u$, $e_{ad}:=u_{ad}-u$, and
$e_{a}^2:=u_{a}^2-{\cal L}_{ma} u$. Since ${\cal
  L}_{ma}u(\cdot,0)=f(\cdot,0)-2ad_xh+a^2h/\nu+\nu d_x^2h$, the
equations for the error are
$$
  \left\{\begin{array}{l}
    {\cal L}_{a}e_a^1 =-\nu\partial_{x}^2u \mbox{ in } \Omega_2,\\
    e_a^1(L_2,\cdot)=0,\\
    e_a^1(\cdot,0)=0,
  \end{array}\right.
  \left\{\begin{array}{l}
    {\cal L}_{a}e_{a}^2 =\cR\,e_a^1  \mbox{ in } \Omega_2,\\
    e_{a}^2(L_2,\cdot)={\cal L}_{ma}e_a^1(L_2,\cdot),\\
    e_{a}^2(\cdot,0)= -\nu d_x^2h,\\
  \end{array}\right.
  \left\{\begin{array}{l}
    {\cal L}_{ad}e_{ad} =0 \mbox{ in } \Omega_1,\\
    e_{ad}(-L_1,\cdot)=0,\\
    {\cal L}_{ma}e_{ad}(0,\cdot)=e_{a}^2(0,\cdot),\\
    e_{ad}(\cdot,0)=0.
  \end{array}\right.
$$
We now analyze each of the three solves separately:

{\bf First advection equation in $\Omega_2$:} With
Theorem~\ref{th:hyperbolic1}, we find that the error $e_a^1$ satisfies
for $k=0,1 \mbox{ and } 2$ the estimate
\begin{equation}\label{eq:errestglobaneg1}
 \|\partial_t^k e_a^1\|_{L^2_{x,t}}^2
 +|a| \|\partial_t^k e_a^1(0,\cdot)\|_{L^2_{t}}^2
 \le
 C\left( \nu^2
 \|\partial_t^k\partial_{x}^2u\|^2_{L^2_{x,t}}
 +\|\partial_t^ke_a^1(\cdot,0)\|_{L^2_x}^2\right).
\end{equation}
The case $k=0$ yields the first result of the theorem. We further compute
$$
  \partial_te_a^1(\cdot,0)=-\nu d_x^2h, \hspace{0.5cm}
  \partial_t^2e_a^1(\cdot,0)=\nu(ad_x+c)d_x^2h-\nu d_x^2 \partial_tu(\cdot,0),
$$
with $\partial_tu(\cdot,0)=f(\cdot,0)-(ad_xh+ch-\nu d_{x}^2h)$, so that
$$
  \|\partial_te_a^1(\cdot,0)\|_{L^2_x}^2\le \nu^2 \|h\|_{H^2_{x}},\hspace{0.5cm}
  \|\partial_t^2e_a^1(\cdot,0)\|^2\le \nu^2(\|f(\cdot,0)\|_{H^{2}_{x}}^2+\|h\|_{H^4_x}^2).
$$
We thus obtain for $\cR\,e_a^1=(c+\partial_t)^2e_a^1$ the estimate
\begin{equation}\label{eq:aneg:Rea1}
  \|\cR\,e_a^1\|_{L^2_{x,t}}^2\le C\nu^2(\|u\|_{H^{2,2}_{x,t}}^2
    + \|f(\cdot,0)\|_{H^{2}_{x}}^2+\|h\|_{H^4_x}^2).
\end{equation}

{\bf Second advection equation in $\Omega_2$:} Using again 
Theorem~\ref{th:hyperbolic1}, we obtain the estimate
$$
  \|e_a^2(0,\cdot)\|_{L^2_t}^2  \le C( \|\cR\,e_a^1\|_{L^2_{x,t}}^2+\nu^2\|h\|_{H^2_x}^2
   +\|{\cal L}_{ma}e_a^1(L_2,\cdot)\|_{L^2_t}^2).
$$
To evaluate ${\cal L}_{ma}e_a^1(L_2,\cdot)$, we observe that
$u(L_2,\cdot)=u_a(L_2,\cdot)$, so that we have ${\cal
  L}_{ma}e_a^1(L_2,\cdot)=-a\partial_xe_a^1(L_2,\cdot)=\nu\partial_{x}^2u(L_2,\cdot)$.
Therefore, using \eqref{eq:aneg:Rea1}, we get
\begin{equation}\label{eq:ea2surlebord}
  \|e_a^2(0,\cdot)\|_{L^2_t}^2 \le C \nu^2(\|u\|_{H^{2,2}_{x,t}}^2
   +\|f(\cdot,0)\|_{H^{2}_{x}}^2+\|h\|_{H^4_x}^2+\|\partial_x^2u(L_2,\cdot)\|_{L^2_t}^2).
\end{equation}

{\bf Advection-diffusion equation in $\Omega_1$:} With Lemma
\ref{lemma:advdifferreuraneg} we obtain
$$
  \|e_{ad}\|_{L^2_{x,t}}^2 \le C \nu^2 \|e_a^2(0,\cdot)\|_{L^2_t}^2.
$$
We can thus conclude using \eqref{eq:ea2surlebord}.

It remains to estimates $\|\partial_x^2u(L_2,\cdot)\|_{L^2_t}$ and
$\|\partial_x^2u \|_{L^2(\Omega_2\times(0,T))}$. If the data is
compactly supported, there is only one boundary layer, at $x=-L_1$,
and (see \cite{Metivier})
$$
  u(t,x)= U(t,x) + e^{a(x+L_1)/\nu} U(t,0) +{\cal O}(\nu),\quad
  \mbox{(note $a<0$)}.
$$
Here, $U$ is the solution of the advection equation in $\Omega$ with
data $g_2$ at $x=L_2$.  The norm of $\partial_{xx} u$, though not
bounded in the entire interval $\Omega$, is bounded in $\Omega_2$, since
$$
  \|d_x^2 e^{a(x+L_1)/\nu} \|_{L^2(\Omega_2)}^2
  =\frac{|a|^3}{2\nu^3} (e^{2a L_1/\nu}-e^{2a(L_2+L_1)/\nu})
  \sim \frac{|a|^3}{2\nu^3} e^{2a L_1/\nu},
$$
which tends to zero as $\nu$ goes to zero, because $a<0$.
Similarly the value at $L_2$ is bounded.
\end{proof}

\section{Numerical Experiments}\label{SecNum}

We use a Crank-Nicolson scheme for the advection-diffusion equation
and an implicit upwind scheme for the advection equation. We
discretize $\Omega:=(-1,1)$ with $N=64000$ points, which leads to a
spatial step $\Delta x=3.125\times 10^{-5}$ and the time step $\Delta
t=\Delta x$. We choose $c=1$, $g_1 \equiv g_2 \equiv 0$, $T=1$ and the
right hand side, shown in Figure \ref{fig:fxt} on the left, is
$$
  \begin{aligned}
    f(x,t)&=f_1(t)f_2(x,t),\\
    f_1(t)&=(\sin^4(4\pi (t-t_0))+ \sin^4(2\pi(t-t_0))/2)\chi_{t>t_0},
      \quad t_0=0.1,\\
    f_2(x,t)&= e^{-100 x^2/4}+e^{-100(x-t/4-0.4)^2}+e^{-100(x+t/2+0.4)^2}.
  \end{aligned}
$$
\begin{figure}
  \centering
  \includegraphics[width=6cm]{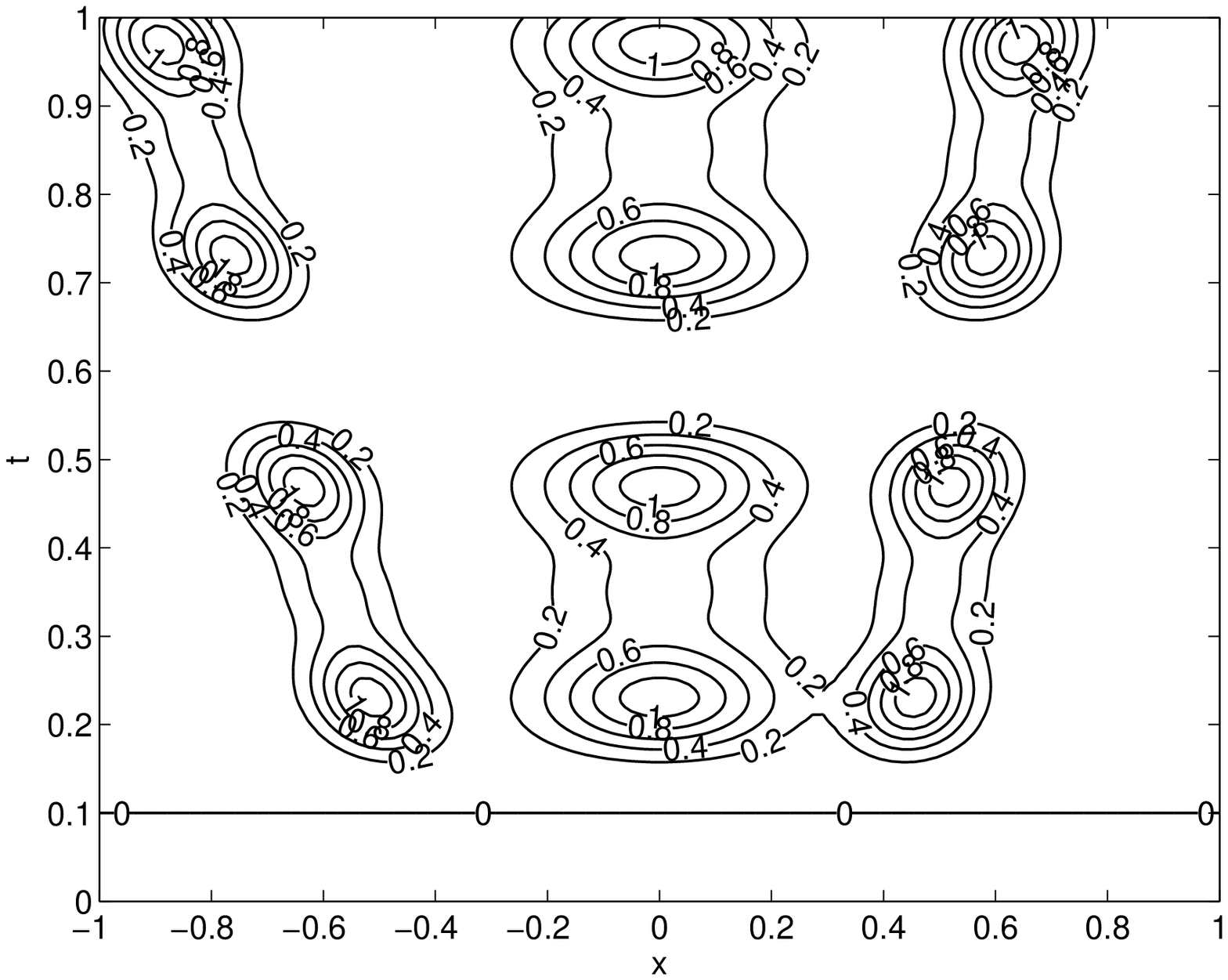}
  \includegraphics[width=6cm]{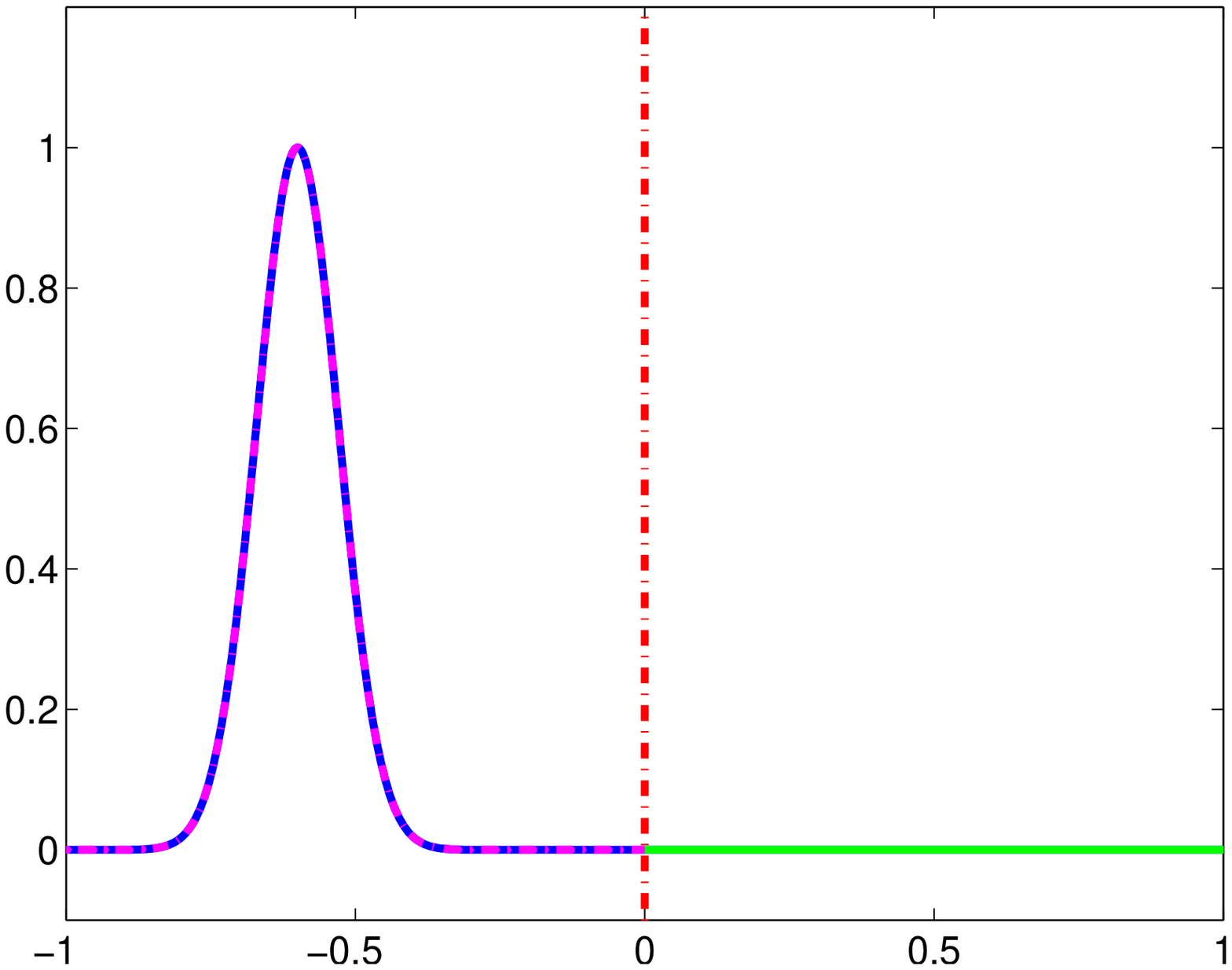}
  \caption{Left: contour plot for the right hand side in space and
    time. Right: initial condition $a>0$}
  \label{fig:fxt}
\end{figure}

\subsection{Positive advection}

We choose $a=1$, with the initial condition, shown in Figure \ref{fig:fxt}
on the right,
$$
  u_0(x)=e^{-100(x-x_0)^2}, \mbox{ with }x_0=-0.6.
$$
Figure \ref{fig:sol} 
\begin{figure}
  \centering
    \mbox{\includegraphics[width=0.33\textwidth]{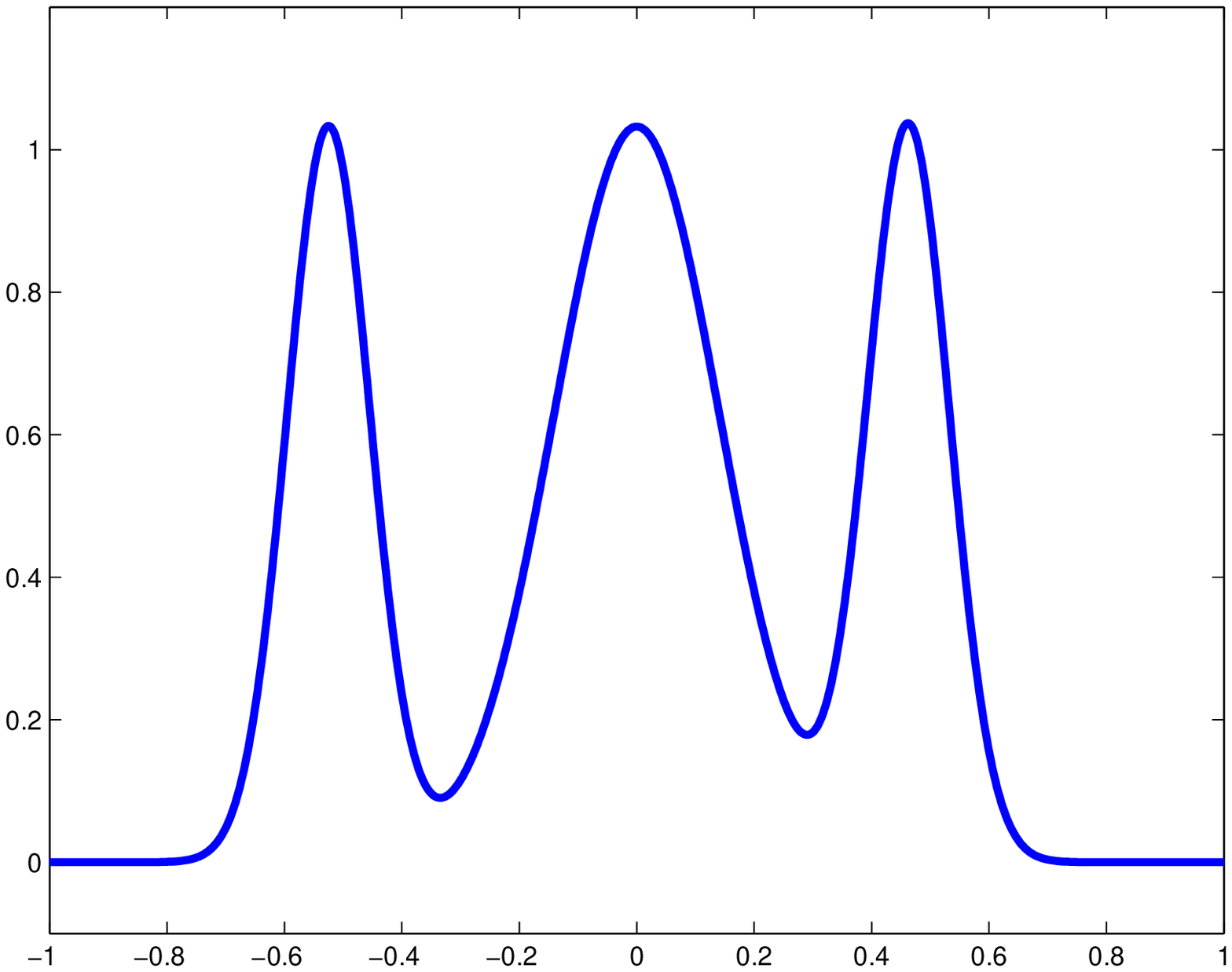}
    \includegraphics[width=0.33\textwidth]{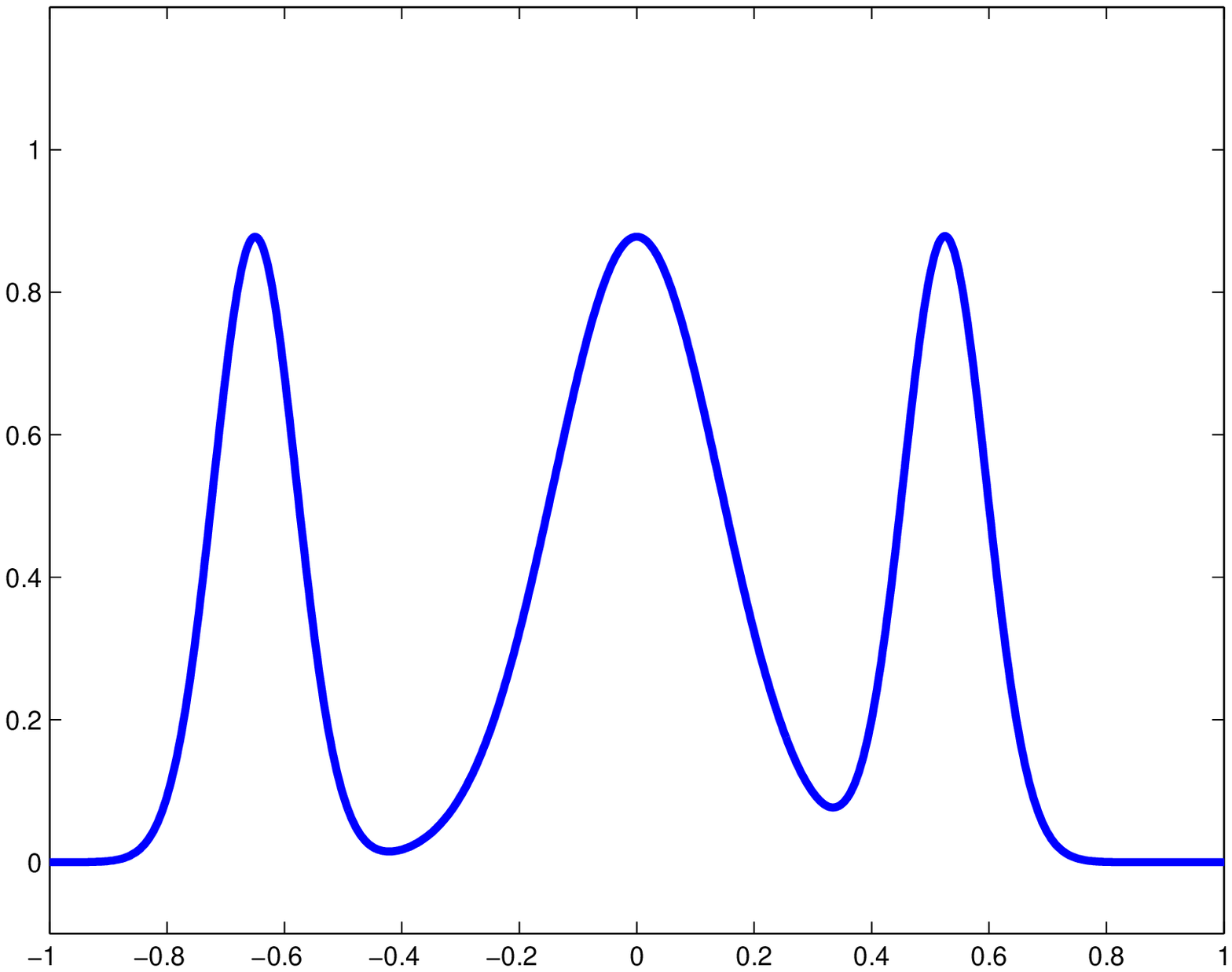}
    \includegraphics[width=0.33\textwidth]{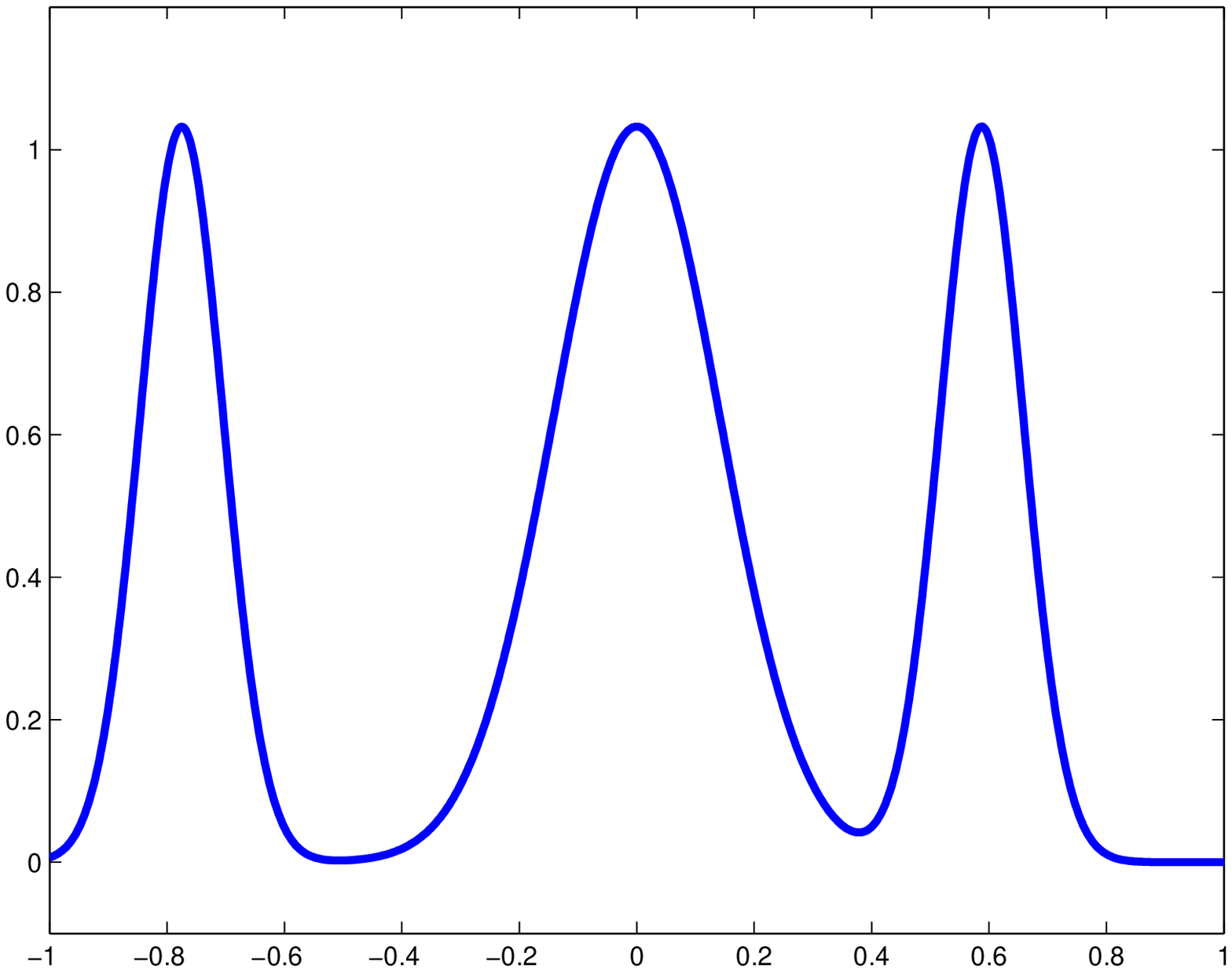}}
    \mbox{\includegraphics[width=0.33\textwidth]{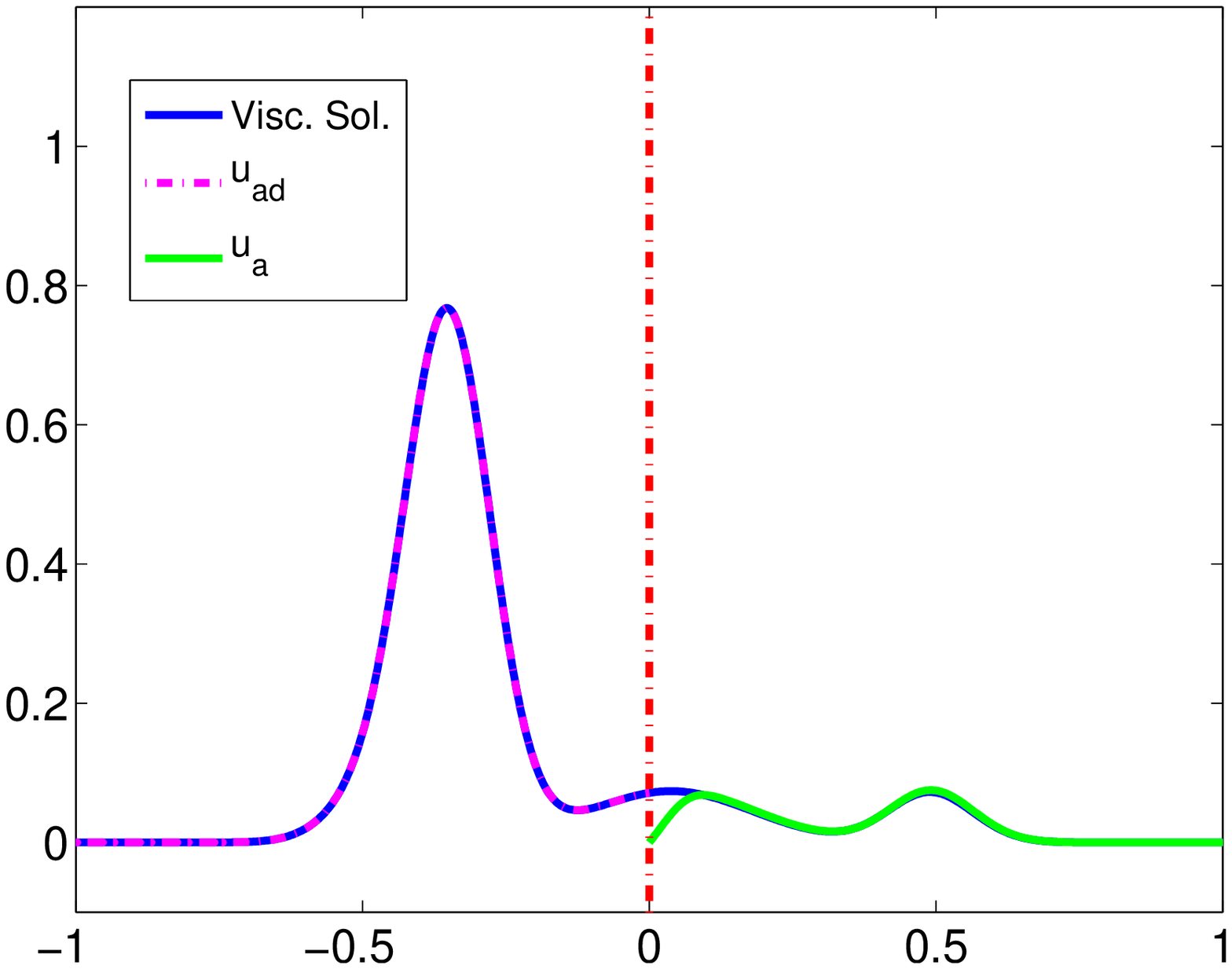}
    \includegraphics[width=0.33\textwidth]{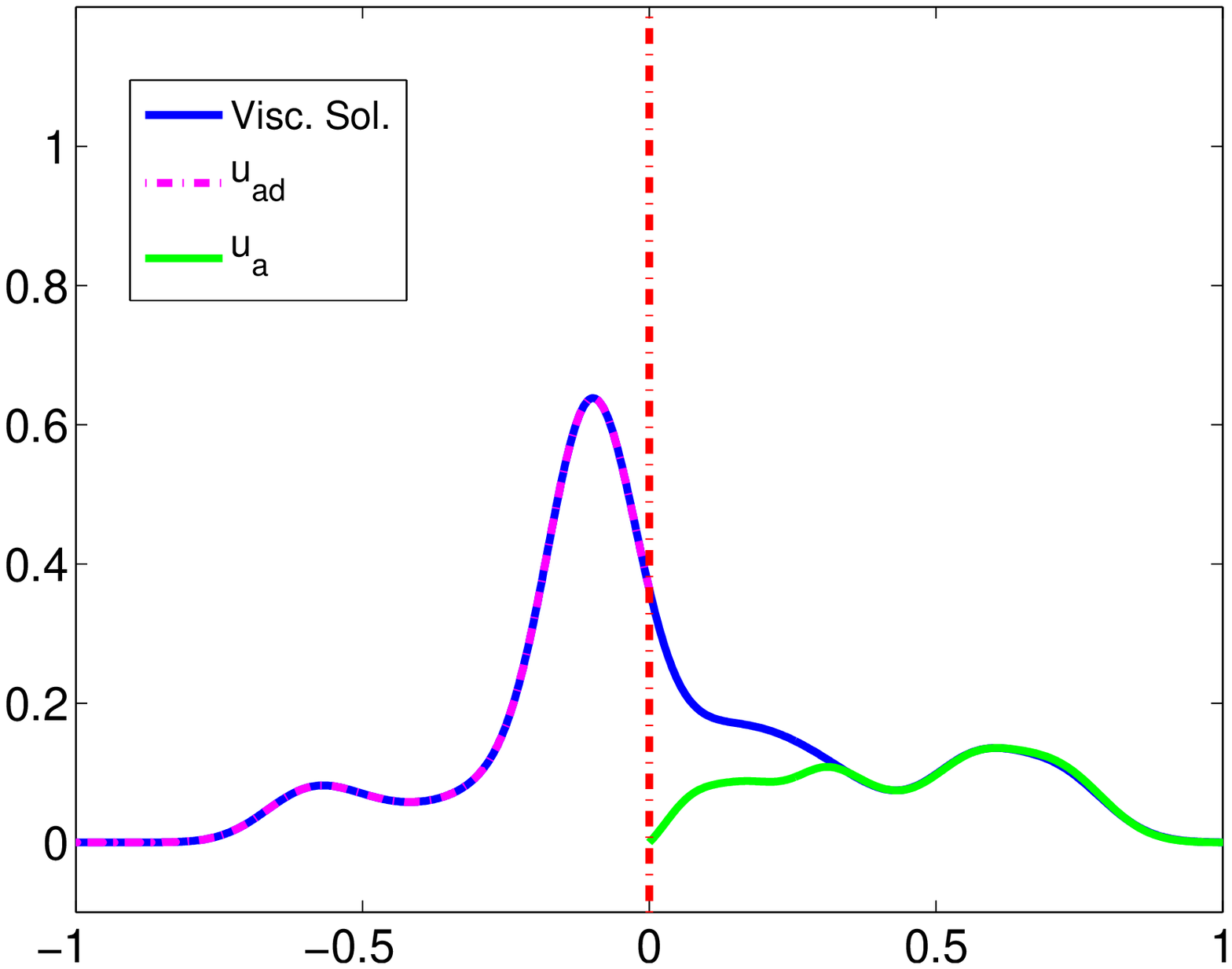}
    \includegraphics[width=0.33\textwidth]{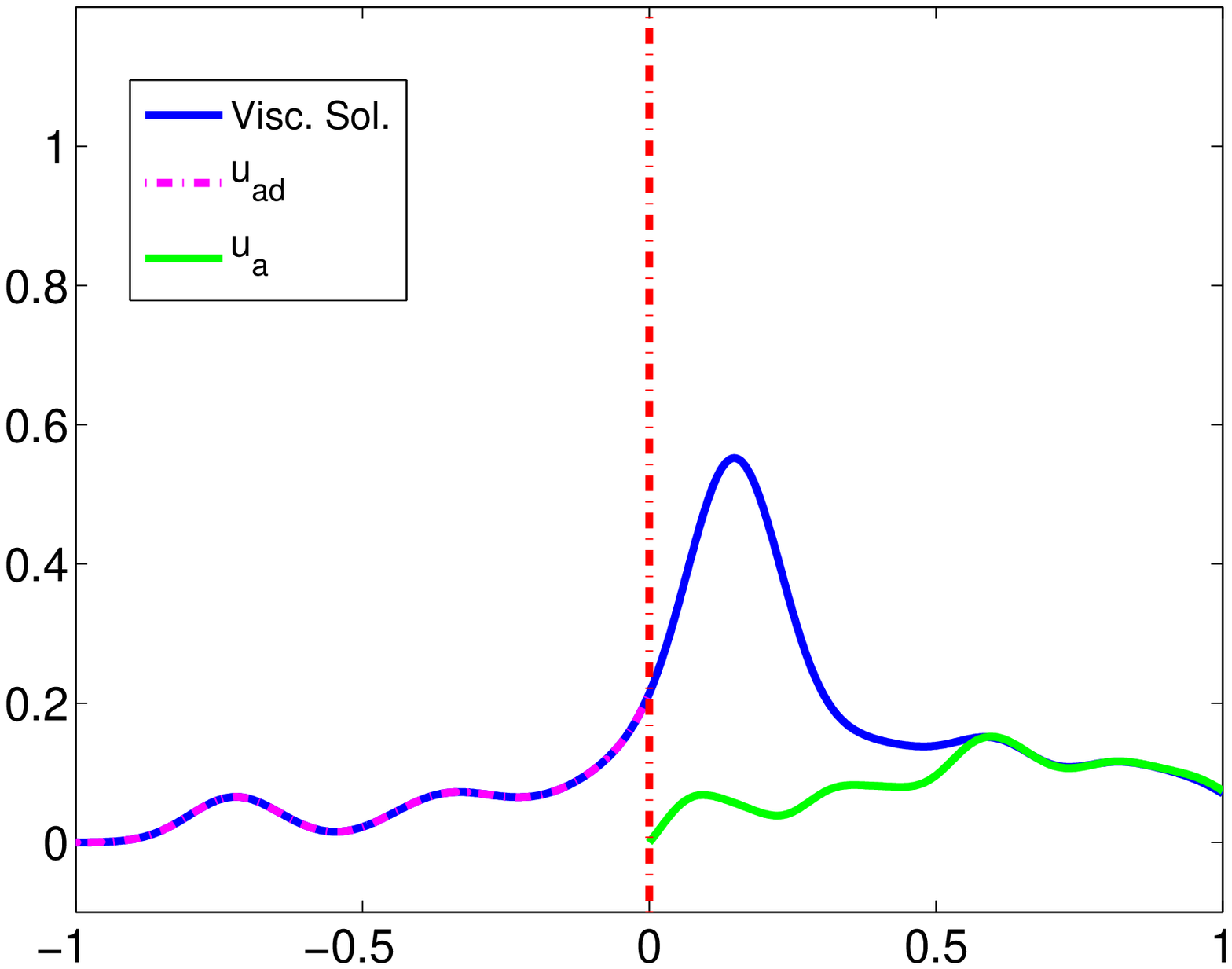}}
    \mbox{\includegraphics[width=0.33\textwidth]{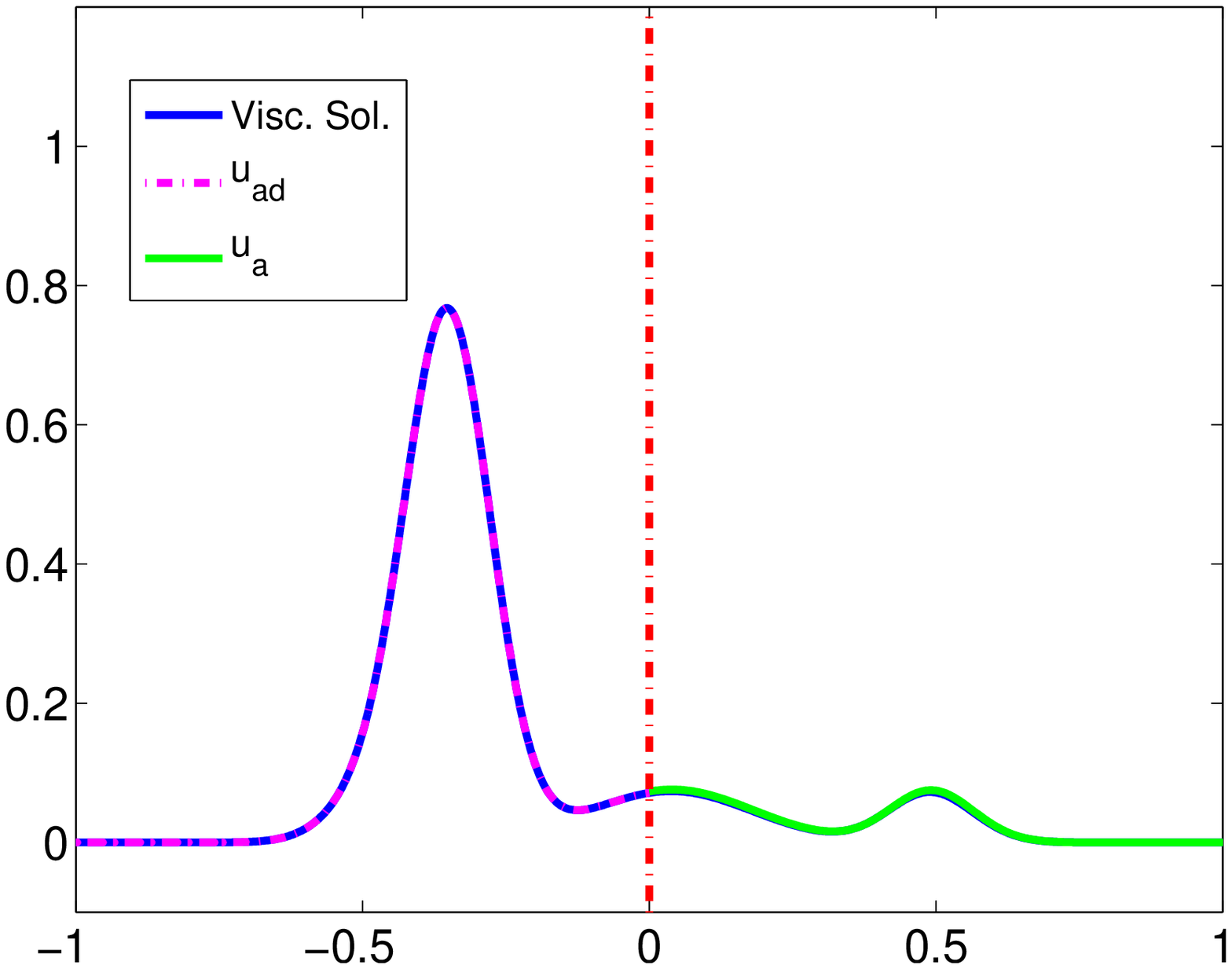}
    \includegraphics[width=0.33\textwidth]{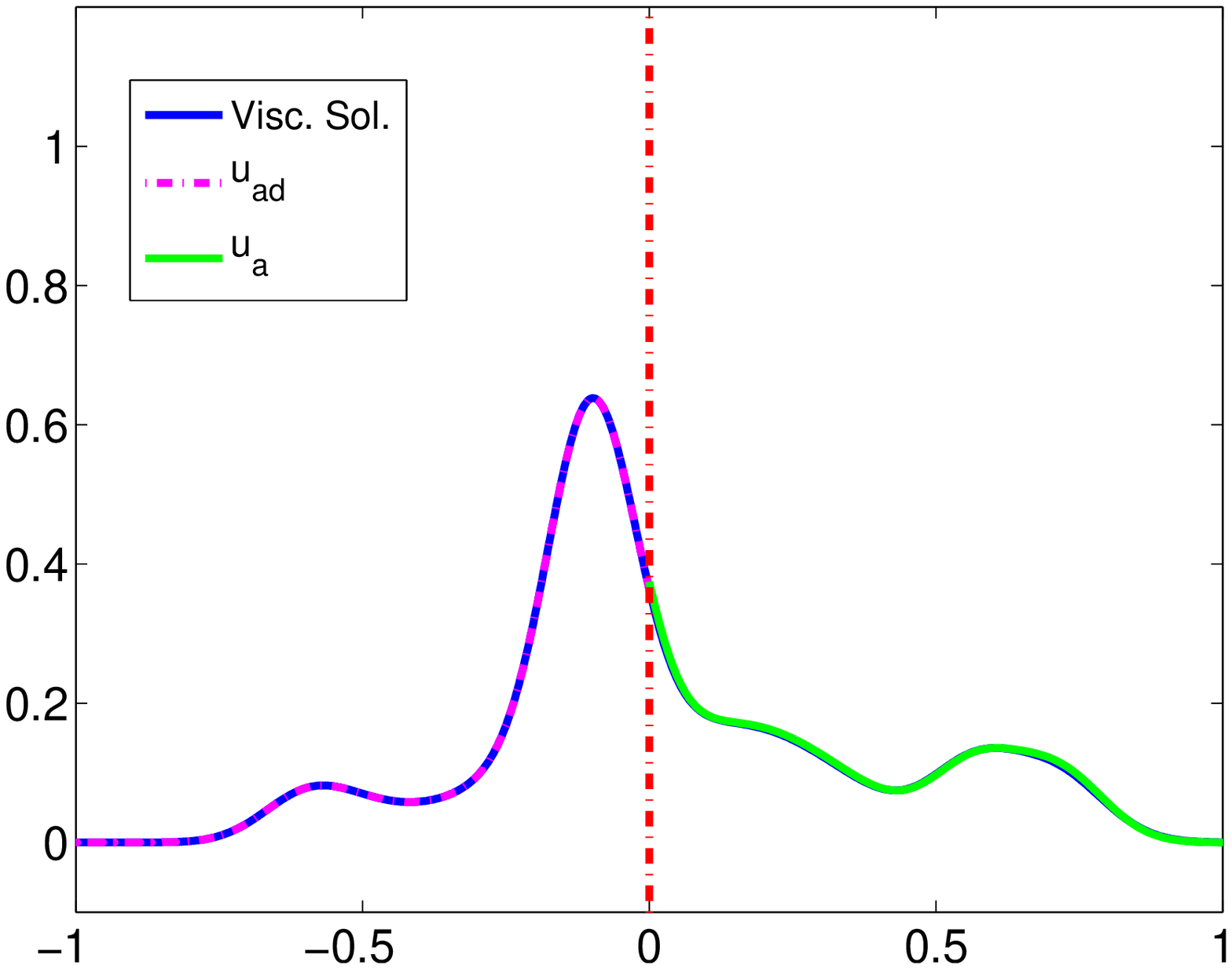}
    \includegraphics[width=0.33\textwidth]{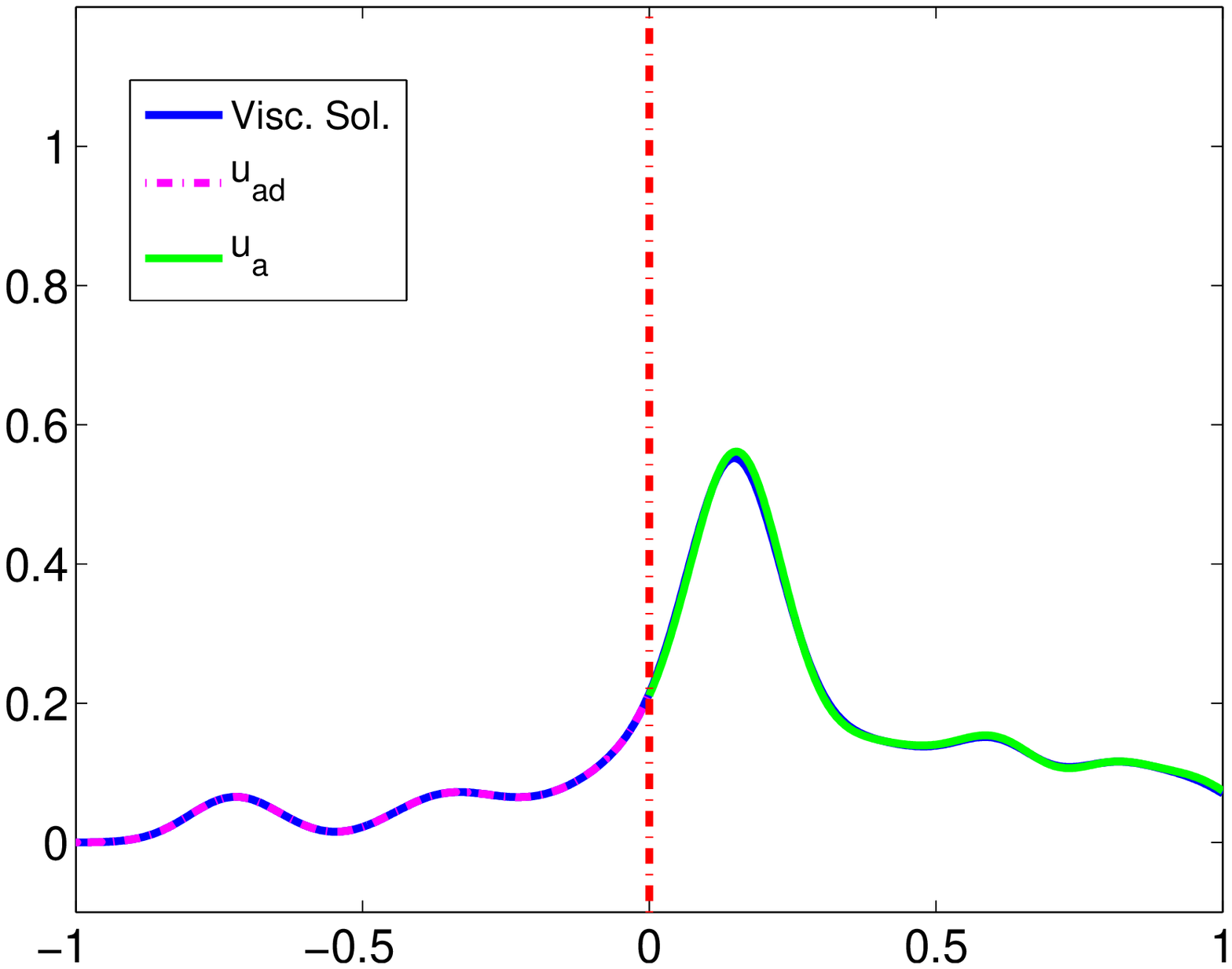}}
  \caption{From left to right: snapshots at time $t=0.25$, $0.5$ and
    $0.75$. First line: right hand side. Second line: solution of
    Algorithm \eqref{algoapos} at iteration $k=1$. Third line:
    solution of Algorithm \eqref{algoapos} at iteration $k=2$ }
  \label{fig:sol}
\end{figure}
shows first snapshots in time of the right hand side, and then of the
viscous solution \eqref{eqadvisqapos} and the solution obtained by the
factorization algorithm~\eqref{algoapos} after one and two iterations
when $\nu=10^{-3}$.  We see that in the first iteration the solution
$u_{ad}^1$ is very close to the viscous solution. This solution is
improved with the second iteration when $u_a$ is also improved.

Figure~\ref{fig:errapos} 
\begin{figure}
  \centering
  \includegraphics[width=0.49\textwidth]{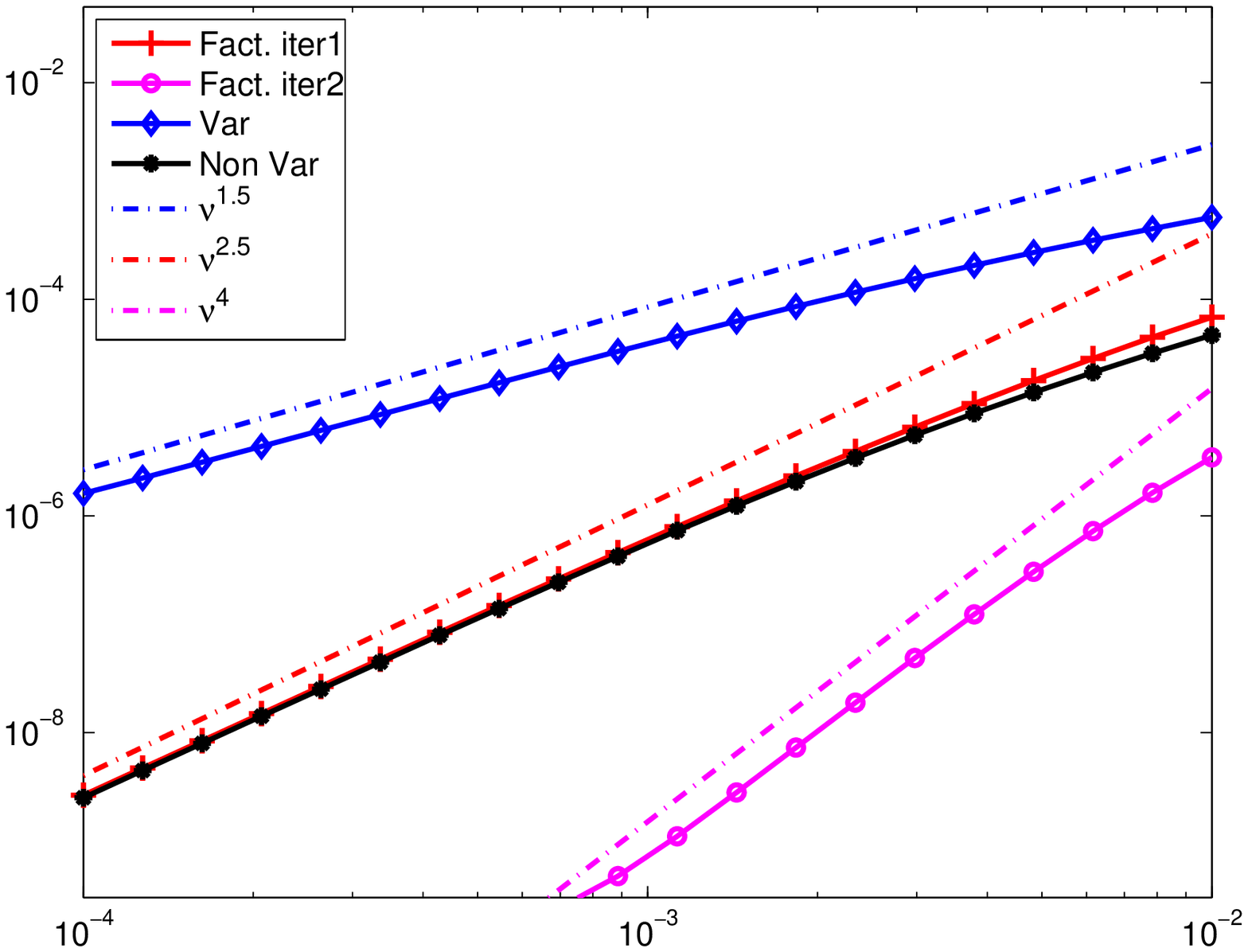}
  \includegraphics[width=0.49\textwidth]{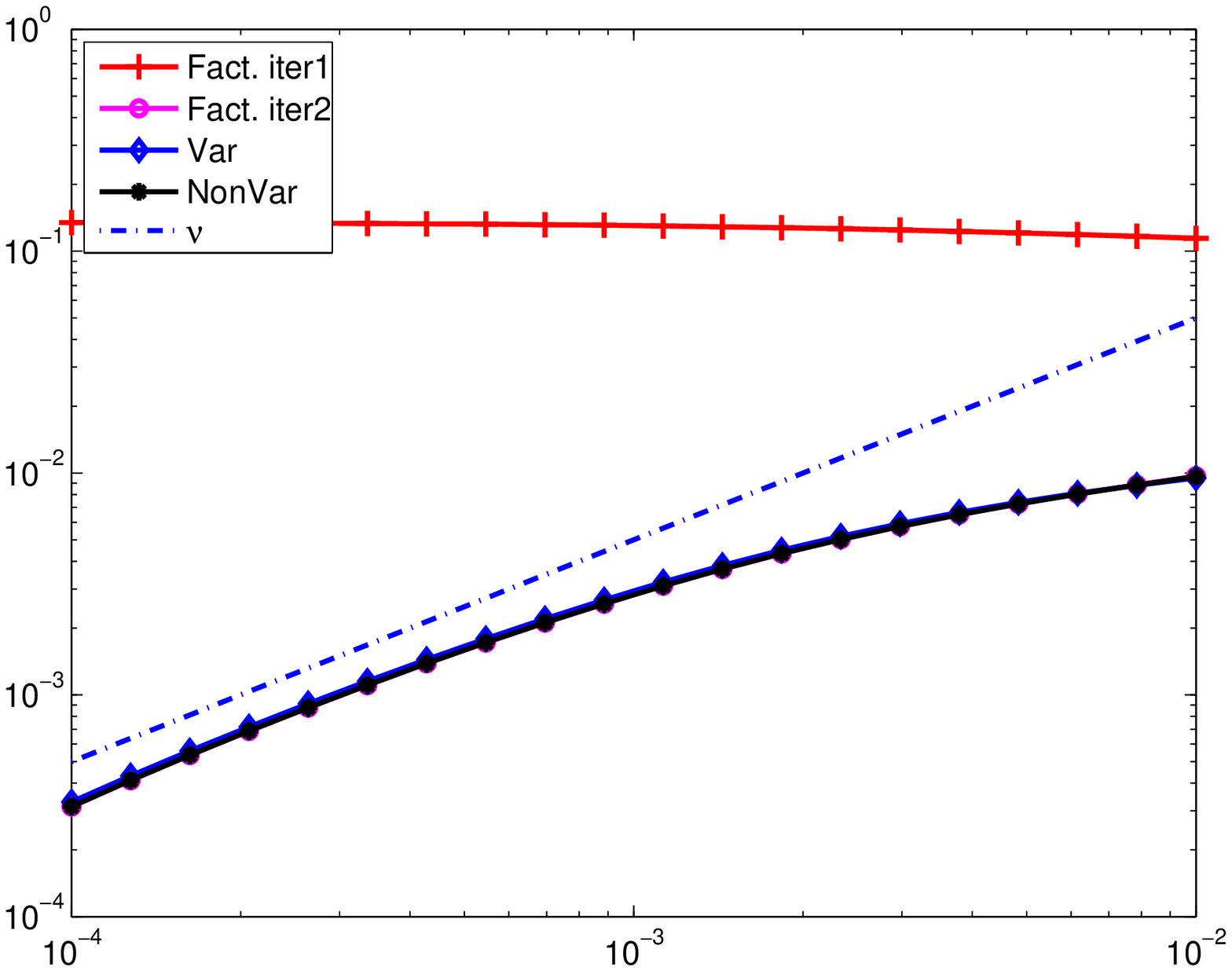}
  \caption{Errors for $a>0$ as the viscosity becomes small for our
    factorization algorithm compared to other coupling algorithms from
    the literature. Left: $\|u-u_{ad}\|_{L^2_{x,t}}$. Right:
    $\|u-u_{a}\|_{L^2_{x,t}}$}
  \label{fig:errapos}
\end{figure}
shows the $L^2$ space-time error as a function of the viscosity
becoming small for the factorization algorithm \eqref{algoapos} and
gives a comparison to algorithms from the literature. These algorithms
solve an advection-diffusion equation ${\cal L}_{ad} u_{ad} =f$ in
$\Omega_1$ and an advection equation ${\cal L}_{a} u_{a} =f$ in
$\Omega_2$, and use for $a>0$ either non-variational transmission
conditions $\partial_xu_{ad}(0,\cdot)=\partial_xu_{a}(0,\cdot)$ and
$u_{ad}(0,\cdot)=u_{a}(0,\cdot)$, see
\cite{Dubach:1993:CRE,Gastaldi:1989:OTC}, or variational transmission
conditions $\nu \partial_xu_{ad}(0,\cdot)=0$ and
$u_{ad}(0,\cdot)=u_{a}(0,\cdot)$, see
\cite{Gastaldi:1989:OTC,Gastaldi:1990:OCT}. We see that the
variational transmission conditions do not need an iteration in this
case, one can first solve advection-diffusion, and then advection. The
error is however ${\cal O}(\nu^{\frac{3}{2}})$ in the viscous region
$\Omega_1$. With only one iteration of the factorization algorithm,
the error is ${\cal O}(\nu^{\frac{5}{2}})$, and with two iterations we
get ${\cal O}(\nu^4)$, both corresponding to our theoretical results
in Theorem \ref{th:theresult}. The non-variational transmission
conditions also give an error ${\cal O}(\nu^{\frac{5}{2}})$, as good
as with one iteration of the factorization algorithm, but one needs to
iterate and choosing a good relaxation parameter to ensure convergence
is not easy; we chose heuristically $\theta=\frac{1}{450\sqrt{\nu}}$
in our computations.  In the inviscid subregion $\Omega_2$, the error
of all methods is ${\cal O}(\nu)$, only the initialization step in the
factorization algorithm has an error of ${\cal O}(1)$, as predicted by
Theorem \ref{th:theresult}.

\subsection{Negative advection}

We now consider a negative advection example, $a=-1$, with 
initial condition
$$
  u_0(x)=e^{-100(x-x_0)^2},\mbox{ with } x_0=0.5.
$$
\begin{figure}
  \centering
  \includegraphics[width=0.49\textwidth]{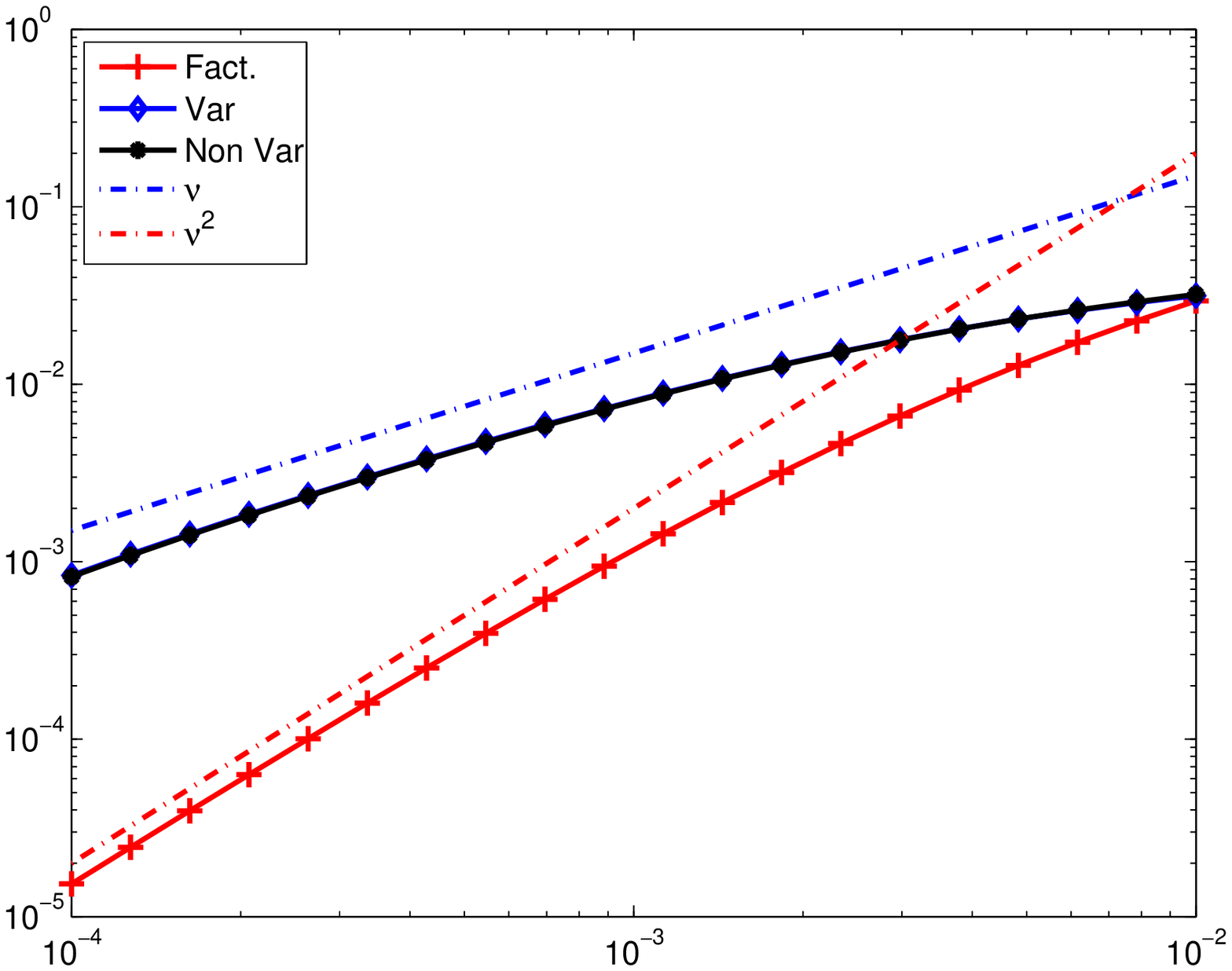}
  \includegraphics[width=0.49\textwidth]{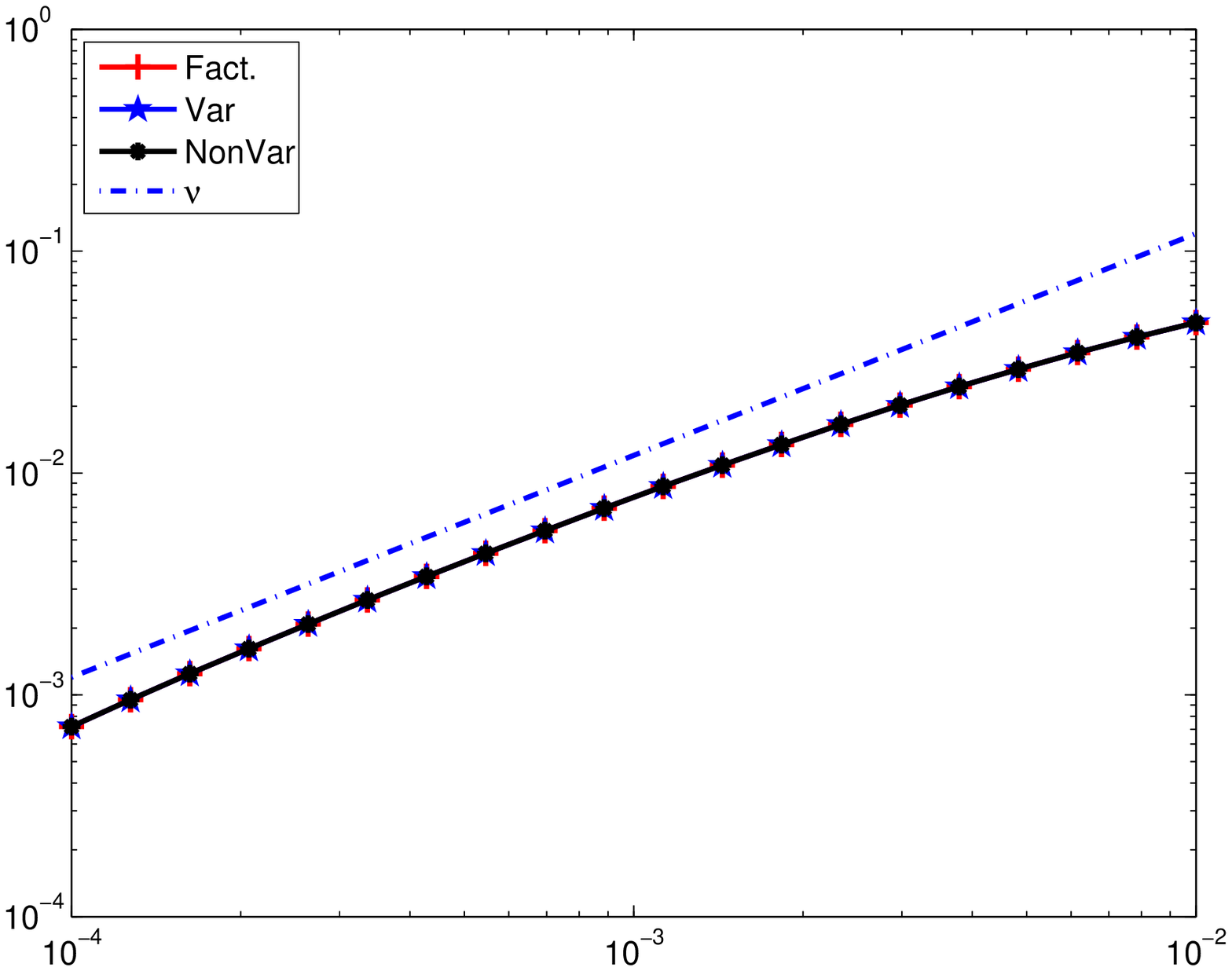}
  \caption{Errors for $a<0$ as the viscosity becomes small for the factorization algorithm compared to other coupling algorithms from the literature. Left: $\|u-u_{ad}\|_{L^2_{x,t}}$. Right: $\|u-u_{a}\|_{L^2_{x,t}}$}
\label{fig:erraneg}
\end{figure}
Figure \ref{fig:erraneg} shows the $L^2$ space-time error
between the viscous solution and the solution of the factorization
algorithm \eqref{algoaneg}, and also a comparison to the errors of the other
coupling algorithms from the literature; the variational coupling
conditions for $a<0$ are
$-\nu\partial_xu_{ad}(0,\cdot)+au_{ad}(0,\cdot)=au_{a}(0,\cdot)$, and
the non-variational ones are $u_{ad}(0,\cdot)=u_{a}(0,\cdot)$. Once
again the error in $\Omega_2$ is ${\cal O}(\nu)$ for each algorithm,
since each algorithm solves the same advection equation in
$\Omega_2$. However the factorization algorithm solves then a second
advection equation which provides a better boundary value for the
advection-diffusion problem in $\Omega_1$ and thus can provide an
error ${\cal O}(\nu^2)$, whereas the other algorithms only give an
approximation ${\cal O}(\nu)$ in $\Omega_1$.

\section{Conclusions}

We introduced a new algorithm to solve advection diffusion problems
with pure advection approximation in a subregion. We call this
algorithm factorization algorithm, because it is based on a
factorization of the underlying operator. We proved rigorous error
estimates that show that our new algorithm gives solutions that are
closer to the fully viscous solution of interest than other coupling
algorithms in the literature. Our numerical experiments indicate that
our estimates are sharp, an issue we are currently investigating using
multiscale expansions.

\bibliographystyle{siam} \bibliography{paper}

\end{document}